\PassOptionsToPackage{unicode}{hyperref}
\PassOptionsToPackage{hyphens}{url}
\PassOptionsToPackage{dvipsnames,svgnames,x11names}{xcolor}
\documentclass[
  letterpaper,
  10pt,
  conference]{ieeeconf}

\usepackage{amsmath,amssymb}
\usepackage{iftex}
\ifPDFTeX
  \usepackage[T1]{fontenc}
  \usepackage[utf8]{inputenc}
  \usepackage{textcomp} 
\else 
  \usepackage{unicode-math}
  \defaultfontfeatures{Scale=MatchLowercase}
  \defaultfontfeatures[\rmfamily]{Ligatures=TeX,Scale=1}
\fi
\usepackage[]{times}
\ifPDFTeX\else  
\fi
\IfFileExists{upquote.sty}{\usepackage{upquote}}{}
\IfFileExists{microtype.sty}{
  \usepackage[]{microtype}
  \UseMicrotypeSet[protrusion]{basicmath} 
}{}
\usepackage{xcolor}
\makeatletter
\ifx\paragraph\undefined\else
  \let\oldparagraph\paragraph
  \renewcommand{\paragraph}{
    \@ifstar
      \xxxParagraphStar
      \xxxParagraphNoStar
  }
  \newcommand{\xxxParagraphStar}[1]{\oldparagraph*{#1}\mbox{}}
  \newcommand{\xxxParagraphNoStar}[1]{\oldparagraph{#1}\mbox{}}
\fi
\ifx\subparagraph\undefined\else
  \let\oldsubparagraph\subparagraph
  \renewcommand{\subparagraph}{
    \@ifstar
      \xxxSubParagraphStar
      \xxxSubParagraphNoStar
  }
  \newcommand{\xxxSubParagraphStar}[1]{\oldsubparagraph*{#1}\mbox{}}
  \newcommand{\xxxSubParagraphNoStar}[1]{\oldsubparagraph{#1}\mbox{}}
\fi
\makeatother
\providecommand{\tightlist}{%
  \setlength{\itemsep}{0pt}\setlength{\parskip}{0pt}}\usepackage{longtable,booktabs,array}
\usepackage{calc} 
\usepackage{etoolbox}
\makeatletter
\patchcmd\longtable{\par}{\if@noskipsec\mbox{}\fi\par}{}{}
\makeatother
\IfFileExists{footnotehyper.sty}{\usepackage{footnotehyper}}{\usepackage{footnote}}
\makesavenoteenv{longtable}
\usepackage{graphicx}
\makeatletter
\def\maxwidth{\ifdim\Gin@nat@width>\linewidth\linewidth\else\Gin@nat@width\fi}
\def\maxheight{\ifdim\Gin@nat@height>\textheight\textheight\else\Gin@nat@height\fi}
\makeatother
\setkeys{Gin}{width=\maxwidth,height=\maxheight,keepaspectratio}
\makeatletter
\def\fps@figure{htbp}
\makeatother
\NewDocumentCommand\citeproctext{}{}

\makeatletter
 \let\@cite@ofmt\@firstofone
 \def\@biblabel#1{}
 \def\@cite#1#2{{#1\if@tempswa , #2\fi}}
\makeatother
\newlength{\cslhangindent}
\newlength{\csllabelwidth}
\newenvironment{CSLReferences}[2] 
 {\begin{list}{}{%
  \setlength{\itemindent}{0pt}
  \setlength{\leftmargin}{0pt}
  \setlength{\parsep}{0pt}
  \ifodd #1
   \setlength{\leftmargin}{\cslhangindent}
   \setlength{\itemindent}{-1\cslhangindent}
  \fi
  \setlength{\itemsep}{#2\baselineskip}}}
 {\end{list}}
\usepackage{calc}

\newcommand{\CSLLeftMargin}[1]{\parbox[t]{\csllabelwidth}{\strut#1\strut}}
\newcommand{\CSLRightInline}[1]{\parbox[t]{\linewidth - \csllabelwidth}{\strut#1\strut}}

\IEEEoverridecommandlockouts
\usepackage{booktabs}
\usepackage{amsmath,amssymb,amsfonts}
\usepackage{xcolor}

\makeatletter
\AtBeginDocument{%
  \@ifundefined{c@theorem}{}{\counterwithout{theorem}{section}}%
  \@ifundefined{c@lemma}{}{\counterwithout{lemma}{section}}%
  \@ifundefined{c@corollary}{}{\counterwithout{corollary}{section}}%
  \@ifundefined{c@proposition}{}{\counterwithout{proposition}{section}}%
  \@ifundefined{c@conjecture}{}{\counterwithout{conjecture}{section}}%
  \@ifundefined{c@definition}{}{\counterwithout{definition}{section}}%
  \@ifundefined{c@example}{}{\counterwithout{example}{section}}%
  \@ifundefined{c@exercise}{}{\counterwithout{exercise}{section}}%
  \@ifundefined{c@refremark}{}{\counterwithout{refremark}{section}}%
  \@ifundefined{c@refsolution}{}{\counterwithout{refsolution}{section}}%
}
\makeatother
\ifx\oldparagraph\undefined\else\let\paragraph\oldparagraph\fi
\makeatletter
\renewcommand{\@IEEEsectpunct}{.\ \,}
\makeatother
\usepackage{caption}
\usepackage{etoolbox}
\AtBeginEnvironment{CSLReferences}{%
  \footnotesize
  \settowidth{\csllabelwidth}{[99]}%
  \addtolength{\csllabelwidth}{\labelsep}}
\usepackage{booktabs}
\usepackage{longtable}
\usepackage{array}
\usepackage{multirow}
\usepackage{wrapfig}
\usepackage{float}
\usepackage{colortbl}
\usepackage{pdflscape}
\usepackage{tabu}
\usepackage{threeparttable}
\usepackage{threeparttablex}
\usepackage[normalem]{ulem}
\usepackage{makecell}
\usepackage{xcolor}
\makeatletter
\@ifpackageloaded{caption}{}{\usepackage{caption}}
\AtBeginDocument{%
\ifdefined\contentsname
  \renewcommand*\contentsname{Table of contents}
\else
  \newcommand\contentsname{Table of contents}
\fi
\ifdefined\listfigurename
  \renewcommand*\listfigurename{List of Figures}
\else
  \newcommand\listfigurename{List of Figures}
\fi
\ifdefined\listtablename
  \renewcommand*\listtablename{List of Tables}
\else
  \newcommand\listtablename{List of Tables}
\fi
\ifdefined\figurename
  \renewcommand*\figurename{Figure}
\else
  \newcommand\figurename{Figure}
\fi
\ifdefined\tablename
  \renewcommand*\tablename{Table}
\else
  \newcommand\tablename{Table}
\fi
}
\@ifpackageloaded{float}{}{\usepackage{float}}
\floatstyle{ruled}
\@ifundefined{c@chapter}{\newfloat{codelisting}{h}{lop}}{\newfloat{codelisting}{h}{lop}[chapter]}
\floatname{codelisting}{Listing}

\usepackage{amsthm}
\theoremstyle{plain}
\newtheorem{lemma}{Lemma}[section]
\theoremstyle{definition}
\newtheorem{exercise}{Assumption}[section]
\theoremstyle{plain}
\newtheorem{theorem}{Theorem}[section]
\theoremstyle{definition}
\newtheorem{definition}{Definition}[section]
\theoremstyle{plain}
\newtheorem{proposition}{Proposition}[section]
\theoremstyle{plain}
\newtheorem{corollary}{Corollary}[section]
\theoremstyle{remark}
\AtBeginDocument{}

\makeatother
\makeatletter
\@ifpackageloaded{caption}{}{\usepackage{caption}}
\@ifpackageloaded{subcaption}{}{\usepackage{subcaption}}
\makeatother

\ifLuaTeX
  \usepackage{selnolig}  
\fi
\usepackage{bookmark}

\IfFileExists{xurl.sty}{\usepackage{xurl}}{} 
\hypersetup{
  pdftitle={High-Dimensional Regularization of the Spatial Sign Covariance Matrix for Robust Shape Estimation},
  pdfauthor={Jonas Elmerraji\^{}\{1\}, James Spall\^{}\{2\}, and Mateo Díaz\^{}\{3\}},
  colorlinks=true,
  linkcolor={blue},
  filecolor={Maroon},
  citecolor={Blue},
  urlcolor={Blue},
  pdfcreator={LaTeX via pandoc}}

\title{\LARGE\bfseries High-Dimensional Regularization of the Spatial
Sign Covariance Matrix for Robust Shape Estimation}
\author{Jonas Elmerraji\(^{1}\), James Spall\(^{2}\), and Mateo
Díaz\(^{3}\)\thanks{$^{1}$Jonas Elmerraji is with the Whiting School of Engineering, Johns Hopkins University, Baltimore, MD, USA {\tt\small jonaselm@jhu.edu}}\thanks{$^{2}$James Spall is with the Department of Applied Mathematics and Statistics, Johns Hopkins University, Baltimore, MD, USA {\tt\small james.spall@jhuapl.edu}}\thanks{$^{3}$Mateo Díaz is with the Department of Statistics \& Data Science Institute, University of Chicago, Chicago, IL, USA {\tt\small mateodd@uchicago.edu}}}
\date{}

\begin{document}
\maketitle
\begin{abstract}
Covariance estimation is a key component of many applications in system
identification and data-driven control. Although heavy-tailed
distributions may lack a covariance matrix to estimate, the shape matrix
provides a well-defined, scale-free generalization for the broad family
of elliptical distributions. In this setting, practitioners often employ
Tyler's M-estimator (TME), which is defined implicitly and is computed
using a fixed-point iteration. The spatial sign covariance matrix
estimator (SSCM) offers a much simpler alternative: it corresponds to
one iteration of TME. Yet, the SSCM has been largely regarded as an
inferior estimator due to its statistical inconsistency under
fixed-dimensional asymptotics. By contrast, using second-order tools
from random matrix theory, we establish that, under standard
assumptions, SSCM asymptotically dominates TME in Frobenius risk when
dimension and sample size grow proportionally. This advantage arises
from implicit regularization, which reduces variance at the cost of a
bias that vanishes as the dimension grows. Numerical experiments support
the theoretical predictions even at moderate dimensions.
\end{abstract}

\section{INTRODUCTION}\label{introduction}

Covariance matrices characterize the linear dependence structure of
random vectors, which, naturally, makes their estimation central to
system identification, data-driven control, signal processing, and
machine learning {[}1{]}, {[}2{]}. In many practical settings, however,
observations are affected by gross outliers and heavy-tailed noise,
which greatly complicate estimation and have motivated robust extensions
of Kalman filtering {[}3{]}, methods for training neural networks
{[}4{]}, and data sensitivity measures {[}5{]}. With sufficiently heavy
tails, the covariance matrix itself may not even exist, leaving sample
moments without a corresponding finite population covariance matrix to
estimate.

To study dependence in a single framework spanning light- and
heavy-tailed noise, we model observations as randomly scaled Gaussian
vectors. This representation encompasses many common noise models used
in control and state estimation, including the Gaussian itself, as well
as Student's \(t\), elliptically symmetric \(\alpha\)-stable, and
contaminated Gaussian noise models. Formally, we consider observations
of the form

\begin{equation}\phantomsection\label{eq-model}{
\begin{aligned}
\mathbf X_i &= \pmb \mu +\sqrt{W_i}\mathbf G_i, & i = 1, \dots, n,
\end{aligned} 
}\end{equation} where each observation \(\mathbf X_i \in \mathbb R^p\)
is a column vector whose corresponding \(\mathbf G_i\) is i.i.d.
\(N_p(\mathbf 0, \pmb \Sigma)\), and the mixing variable \(W_i>0\) is a
scalar random variable. Our goal is to estimate \(\pmb\Sigma\), the
covariance matrix of the latent Gaussian vector \(\mathbf G_i\), which
provides a well-defined estimation target even when the observations
\(\mathbf X_i\) have infinite variance. We make some blanket assumptions
following {[}6{]}: we suppose that the location vector
\(\pmb \mu \in \mathbb R^p\) is known, and without loss of generality
set it to \(\pmb \mu = \mathbf 0\). Further, we assume \(\pmb\Sigma\) is
positive definite and that \(\text{tr}(\pmb\Sigma) = p\) for
identifiability. Let
\(\mathbf X = [\mathbf X_1 \cdots \mathbf X_n]^\top\) and
\(\mathbf G = [\mathbf G_1 \cdots \mathbf G_n]^\top\) denote the
\(n \times p\) matrices of stacked observations and latent Gaussian
vectors, respectively.

Equation (\ref{eq-model}) includes the elliptical family of
distributions, the standard generalization of the Gaussian to
heavy-tailed noise commonly used in signal processing and state
estimation {[}7{]}. Under the previously mentioned trace-\(p\)
restriction, \(\pmb\Sigma\) is often called a \emph{shape matrix}, as it
defines the scale-free family of ellipsoids that form the level surfaces
of the latent Gaussian density. Even when the population covariance
matrix does not exist for the observations, if \(\mathbf X_i\) is itself
elliptical and its density exists, the shape matrix likewise defines a
representative of its ellipsoidal level surfaces where the p.d.f.
remains constant. One such example of the infinite-variance case is the
elliptical \(\alpha\)-stable distribution with \(\alpha < 2\).

While \(\mathbf G_i\) are themselves assumed to be i.i.d., no
independence need be assumed for the mixing variables. Because of this,
when the mixing variables are allowed to depend on the Gaussian vectors
or on one another, (\ref{eq-model}) further encompasses substantially
non-elliptical distributions. In such cases, \(\pmb\Sigma\) describes
the dependence imposed by the latent Gaussian vectors, but not any
additional dependence introduced through the mixing variables. The same
separation applies to scale: the trace restriction on \(\pmb\Sigma\)
removes absolute scale, which is instead carried by the mixing
variables, but the overall scale of \(\mathbf X_i\) can be estimated
separately or supplied externally, for instance from available
conditioning information.

When the law of \(\mathbf X_i\) is fully specified, maximum likelihood
estimators (MLEs) are a natural choice for estimating \(\pmb\Sigma\), as
they are asymptotically consistent and efficient under standard
regularity conditions. When the specific distribution of the mixing
variable is not known, maximum likelihood can still be used to estimate
the shape matrix; because \(W_i>0\), normalizing each observation to
unit length removes its random scale, leaving directions distributed
according to an angular central Gaussian distribution that depends only
on \(\pmb\Sigma\). One appealing aspect of this approach is that it is
robust to model misspecification within the class of models allowed by
(\ref{eq-model}). When it exists, this maximum likelihood estimator for
the directions of the normalized vectors is called Tyler's M-estimator
(TME), which is defined as the matrix \(\hat{\pmb\Sigma}_{\text{TME}}\)
that satisfies the fixed point equation

\begin{equation}\phantomsection\label{eq-tme}{
\begin{aligned}
\hat{\pmb\Sigma}_{\text{TME}} &= \frac{p}{n}\sum_{i=1}^n \frac{\mathbf X_i \mathbf X_i^\top}{\mathbf X_i^\top \hat{\pmb\Sigma}_{\text{TME}}^{-1}\mathbf X_i}, & \text{tr}(\hat{\pmb\Sigma}_{\text{TME}}) = p.
\end{aligned}
}\end{equation}

Among M-estimators for shape on elliptical data, TME minimizes the
worst-case asymptotic variance with dimension fixed {[}8{]}, {[}9{]},
and is considered by many to be the standard choice for estimating
shape. Equation (\ref{eq-tme}) is typically solved iteratively, using
the inverse of the current estimate to compute the next step. This
procedure becomes expensive in high dimensions, however, where both the
inversion and the number of iterations contribute substantially to the
cost; even approximating TME cheaply remains an active area of research
{[}10{]}.

A much simpler shape estimator is the spatial sign covariance matrix
(SSCM), a robust estimator introduced by {[}11{]} and defined as

\begin{equation}\phantomsection\label{eq-sscm}{
\hat{\pmb\Sigma}_{\text{SSCM}}=\frac{p}{n}\sum_{i=1}^n\frac{\mathbf X_i \mathbf X_i^\top}{\|\mathbf X_i\|^2} = \frac{p}{n}\sum_{i=1}^n\frac{\mathbf X_i \mathbf X_i^\top}{\mathbf X_i^\top\mathbf X_i},
}\end{equation} where \(\|\cdot\|\) is the Euclidean norm. The SSCM can
be seen as either the sample covariance matrix of Euclidean-normalized
observation vectors or as a one-shot version of TME that uses the
identity as its initialization. Note that the SSCM is typically trace 1
by construction; in (\ref{eq-sscm}) we rescale to trace \(p\) for ease
of comparison. Both TME and the SSCM are distribution-free for our
model, in the sense that the mixing variables, including their
dependence, cancel exactly in the fractions of (\ref{eq-tme}) and
(\ref{eq-sscm}). Likewise, observe that, because
\(\mathbf X_i/\|\mathbf X_i\| = \mathbf G_i/\|\mathbf G_i\|\) when
\(\pmb\mu = \mathbf 0\), it suffices to consider Gaussian observations,
which simplify calculations. While TME is the standard choice for robust
shape estimation, the SSCM has effectively been left out of the modern
estimation toolbox, in large part because it is statistically
inconsistent under classical fixed-dimension asymptotics {[}12{]}.
Whether that inconsistency is a defect, however, depends on the
relationship between dimension and sample size.

\paragraph*{Main contribution}\label{main-contribution}
\addcontentsline{toc}{paragraph}{Main contribution}

We show that the SSCM asymptotically outperforms benchmark approaches in
mean-squared error when the dimension \(p\) and number of observations
\(n\) grow to infinity together, and can outperform them at practically
relevant finite dimensions. Specifically, the SSCM's lack of consistency
provides beneficial regularization: it trades practically significant
reductions in variance for a bias cost that vanishes as the dimension
grows.

\section{PRELIMINARIES AND
ASSUMPTIONS}\label{preliminaries-and-assumptions}

Using \(\|\cdot\|_F\) do denote the Frobenius norm, we quantify the
bias--variance tradeoff in shape estimation using Frobenius risk, the
expected sum of squared errors across the entries of a shape estimate:

\[
\text{Risk}(\hat{\pmb\Sigma},\pmb\Sigma ) = \mathbb E\|\hat{\pmb\Sigma} - \pmb\Sigma\|_F^2.
\] Where it aids in risk comparison across problems of varying (but
finite) dimensions, we normalize this risk by \(p\).

\begin{definition}[Asymptotic
Dominance]\protect\hypertarget{def-asymptotic-dominance}{}\label{def-asymptotic-dominance}

Along a sequence of models, we say that \(\hat{\pmb\Sigma}_A\)
\emph{asymptotically dominates} \(\hat{\pmb\Sigma}_B\) if \[
\liminf_{n,p\to\infty}
\left[
\operatorname{Risk}(\hat{\pmb\Sigma}_B,\pmb\Sigma)
-
\operatorname{Risk}(\hat{\pmb\Sigma}_A,\pmb\Sigma)
\right] > 0.
\]

\end{definition}

This comparison retains risk differences that persist asymptotically,
even when common leading-order growth makes the risk ratio tend to one.
These effects can contribute to material differences in approximate risk
at finite dimensions. To study this high-dimensional regime where sample
size and dimension grow together, we introduce the following assumptions
standard in random matrix theory (RMT). Throughout, the observations
satisfy (\ref{eq-model}) and the conditions stated with it. Note that
the numbered assumptions below apply only where explicitly invoked.

Finally, observe that for fixed dimension, Frobenius risk is
proportional to the mean-squared error (MSE), and dominance in one
implies dominance in the other.

\begin{exercise}[High-dimensional
regime]\protect\hypertarget{exr-hd-regime}{}\label{exr-hd-regime}

~

\begin{enumerate}
\def\labelenumi{(\alph{enumi})}
\item
  The dimension \(p=p(n)\) grows with the sample size \(n\) such that
  \(c_p:=p/n\rightarrow c\in(0,\infty)\) as \(n,p\to\infty\).
\item
  The eigenvalues \(\lambda_k(\pmb\Sigma)\) of \(\pmb\Sigma\) are
  uniformly bounded away from zero such that
  \(\lambda_- \leq \lambda_k(\pmb\Sigma) \leq \lambda_+\),
  \(k=1,\ldots,p\), for constants \(\lambda_-, \lambda_+ > 0\)
  independent of \(p\).
\item
  The empirical spectral distribution (ESD) of \(\pmb\Sigma\), which
  assigns equal weight to each eigenvalue,
  \(H_p(x) := \frac{1}{p}\sum_{k=1}^p \mathbf 1{\{\lambda_k(\pmb\Sigma) \leq x\}}\)
  converges weakly to a probability distribution function \(H\) as
  \(p \to \infty\).
\end{enumerate}

\end{exercise}

Additionally, call \(F^{\hat{\pmb\Sigma}}\), the ESD of the estimator
\(\hat{\pmb\Sigma}\), defined analogously to \(H_p\). Note that the ESD
is defined for deterministic and random matrices alike. All matrices and
vectors are understood to belong to sequences indexed by \(n\) and
\(p\); to simplify notation, dependence on \(n\) and \(p\) is suppressed
unless needed for clarity. Thus, for example, \(\hat{\pmb\Sigma}\) and
its corresponding ESD \(F^{\hat{\pmb\Sigma}}\), should be understood as
\(\hat{\pmb\Sigma}^{(p)}\) and \(F^{\hat{\pmb\Sigma}^{(p)}}\),
respectively. For results related to TME, we assume \(c_p< 1\) and
\(c <1\) so that that estimator exists for every scenario. Finally, we
adopt the notation of {[}13{]} to denote the moments
\(\beta_{k,p} := \int x^k\mathrm{d}H_p(x) = p^{-1}\text{tr} (\pmb\Sigma^k)\),
with limit
\(\beta_{k,p} \underset{n,p \rightarrow \infty}{\longrightarrow} \beta_k := \int x^k \mathrm{d}H(x)\).

We are interested in comparing the performance of three categories of
estimators of \(\pmb\Sigma\): unbiased estimators, of which the latent
sample covariance matrix \(\frac{1}{n}\mathbf G^\top \mathbf G\) is a
well-studied example, the spatial sign covariance matrix, and Tyler's
M-estimator. Analysis of the SSCM and TME, as well as their respective
tradeoffs in the conventional setting where \(p\) is fixed and
\(n \rightarrow \infty\) can be found in {[}12{]} and {[}9{]}. While
\(\mathbf G\) is generally not observable and thus the latent sample
covariance matrix is not generally attainable, it provides a useful
``oracle'' benchmark.

The sample covariance matrix, SSCM, and TME are examples of
\emph{rotationally equivariant estimators}. Such estimators,
\(\hat{\pmb\Sigma}=\hat{\pmb\Sigma}(\mathbf X )\), have the property
that for any orthogonal matrix \(\mathbf O \in \mathbb R^{p \times p}\),
\(\hat{\pmb\Sigma}(\mathbf X \mathbf O)=\mathbf O^\top\hat{\pmb\Sigma}(\mathbf X )\mathbf O\).
Thus, an orthogonal transformation of the data induces the corresponding
orthogonal transformation of the estimator, preserving its eigenvalues.

For the Gaussian sample covariance, when \(p\) is fixed, its ESD
asymptotically converges to that of the population \(\pmb \Sigma\):
e.g.~as \(n\rightarrow \infty\) and \(c_p \rightarrow 0\) we have
\(F^{\hat{\pmb\Sigma}} \rightarrow H_p\). For example, when
\(\pmb\Sigma=\mathbf I_p\), all sample eigenvalues converge to one
almost surely. Under Assumption~\ref{exr-hd-regime}, however, these
eigenvalues retain a persistent spread and their ESD converges weakly
almost surely to the Marčenko--Pastur (MP) distribution {[}14{]}, whose
continuous part has density

\[
f_{\text{MP}}(x)=\frac{\sqrt{(b-x)(x-a)}}{2\pi cx},
\qquad a<x<b,
\] where \(a=(1-\sqrt c)^2\) and \(b=(1+\sqrt c)^2\). For \(c>1\), the
distribution also has an atom of mass \(1-1/c\) at zero, reflecting the
rank deficiency of the sample covariance matrix when \(p>n\). Thus, even
for an identity population shape, sampling under the regime described by
Assumption~\ref{exr-hd-regime} produces a systematic distortion in the
estimated eigenvalues. For general \(\pmb\Sigma\), the limiting spectral
distribution of the sample covariance matrix depends on both the aspect
ratio \(c\) and the population spectrum \(H\). We denote this
generalized MP distribution \(F^{c,H}\). Mirroring the RMT literature,
we also make use of the finite-dimensional proxy \(F^{c_p, H_p}\), a
generalized MP distribution corresponding to the finite-\(p\) aspect
ratio \(c_p\) and population ESD \(H_p\); more details can be found in
{[}15{]}.

For data generated by (\ref{eq-model}) under
Assumption~\ref{exr-hd-regime}, the latent Gaussian sample covariance
matrix, SSCM and TME have been found to share a common limiting
generalized MP distribution \(F^{c,H}\) {[}6{]}, {[}13{]}, {[}15{]}.
This means that when comparing these estimators in the high-dimensional
setting, the limiting spectral distribution is insufficient to
distinguish between them even when practical differences are large, and
second-order tools are required to describe asymptotic dominance.

Our comparisons with unbiased estimators require the following
additional restriction on the mixing variables.

\begin{exercise}[Gaussian scale
mixtures]\protect\hypertarget{exr-gsm}{}\label{exr-gsm}

The mixing vector \((W_1,\ldots,W_n)\) is independent of the Gaussian
sample \((\mathbf G_1,\ldots,\mathbf G_n)\), and its distribution does
not depend on \(\pmb\Sigma\).

\end{exercise}

When Assumption~\ref{exr-gsm} holds for data generated according to
(\ref{eq-model}), the resulting observations \(\mathbf X_i\) are
frequently referred to as \emph{Gaussian scale mixtures}. Note that
unbiasedness here means
\(\mathbb{E}_{\pmb\Sigma}[\hat{\pmb\Sigma}]=\pmb\Sigma\) for every
\(\pmb\Sigma\) with \(\operatorname{vec}(\pmb\Sigma) \in\Phi\), where
\(\Phi\) is the parameter space of \(\pmb\Sigma\), represented in
\(\operatorname{vec}\) coordinates

\[
\Phi
:=
\left\{
\operatorname{vec}(\pmb\Sigma):
\pmb\Sigma\in\mathbb S_{++}^p,\ 
\operatorname{tr}(\pmb\Sigma)=p
\right\},
\] and \(\mathbb S_{++}^p\) denotes the cone of \(p\times p\) real
symmetric positive-definite matrices.

Lastly, the following propositions compile some technical details that
are useful for proofs of the main results.

\begin{proposition}[Uniform Integrability of TME-SSCM Quadratic Trace
Gap]\protect\hypertarget{prp-tme-l2}{}\label{prp-tme-l2}

Under the model of (\ref{eq-model}), suppose
Assumption~\ref{exr-hd-regime} holds with \(c_p<1\) and \(c < 1\) and
let
\(\Pi_n:=\operatorname{tr}(\hat{\pmb\Sigma}_{\mathrm{TME}}^2)-\operatorname{tr}(\hat{\pmb\Sigma}_{\mathrm{SSCM}}^2)\).
Then \(\sup_n\mathbb E|\Pi_n|^2<\infty\) and \(\{\Pi_n\}\) is uniformly
integrable.

\end{proposition}

\begin{proposition}[Estimator Bias
Control]\protect\hypertarget{prp-estimator-control}{}\label{prp-estimator-control}

Under the model of (\ref{eq-model}), suppose
Assumption~\ref{exr-hd-regime} holds. Then
\(\|\mathbb E (\hat{\pmb\Sigma}_{\text{SSCM}}) - \pmb\Sigma\|_F^2 = \mathcal O(1/p)\).
When we additionally have \(c_p, c < 1\),
\(\|\mathbb E (\hat{\pmb\Sigma}_{\text{TME}}) - \pmb\Sigma\|_F^2 = \mathcal O(1/p)\)
and
\(\text{tr}(\mathbb E (\hat{\pmb\Sigma}_{\text{TME}})\pmb\Sigma) = \text{tr}(\pmb\Sigma^2) + \mathcal O(1/p)\).

\end{proposition}

Because Proposition~\ref{prp-tme-l2} and
Proposition~\ref{prp-estimator-control} are auxiliary to the main
results and their proofs require comparatively involved
estimator-specific analysis, we provide the details in the online
\hyperref[sec-supplement]{Supplement} to this paper.

\section{MAIN RESULTS}\label{main-results}

Conveniently, for rotationally equivariant estimators, our risk function
of interest can be written in terms of linear spectral statistics (LSS).

\begin{definition}[Centered Linear Spectral
Statistics]\protect\hypertarget{def-centered-lss}{}\label{def-centered-lss}

Let \(\pmb \Gamma \in \mathbb R^{p \times p}\) be a deterministic
positive-semidefinite matrix with ESD \(K_p\). The \emph{centered linear
spectral statistics} of an estimator
\(\hat{\pmb\Sigma} \in \mathbb R^{p \times p}\) are defined as:

\[
L_{\pmb \Gamma}^{\hat{\pmb\Sigma}}(f) := \int f(x)\mathrm{d}L_{\pmb \Gamma}^{\hat{\pmb\Sigma}}(x) = \int f(x)\mathrm{d}[F^{\hat{\pmb\Sigma}}(x)-F^{c_p, K_p}(x)]
\]

where \(f\) is a function on \([0, \infty)\).

\end{definition}

Definition~\ref{def-centered-lss} follows the conventions of Equation
2.10 in {[}6{]}. Next, we write our risk function in terms of LSS.

\begin{lemma}[Risk
Decomposition]\protect\hypertarget{lem-bias-variance}{}\label{lem-bias-variance}

For any positive-semidefinite rotationally equivariant matrix estimator
\(\hat{\pmb\Sigma}\) with \(\mathbb E(\hat{\pmb\Sigma}) = \pmb \Gamma\)
and \(\operatorname{tr}(\pmb\Gamma)=p\), the loss can be expressed as,

\begin{equation}\phantomsection\label{eq-risk-decomp}{
\mathbb E\|\hat{\pmb\Sigma} - \pmb\Sigma\|_F^2 = pc_p + p\mathbb E(L_{\pmb \Gamma}^{\hat{\pmb\Sigma}}(x^2)) + \|\pmb \Gamma - \pmb\Sigma\|_F^2.
}\end{equation}

\begin{proof}
Call \(K_p\) the ESD of \(\pmb \Gamma\). Then we can write:

\[
\begin{aligned}
\mathbb E\|\hat{\pmb\Sigma} - \pmb\Sigma\|_F^2 &=  \mathbb E\text{tr}(\hat{\pmb\Sigma}^2)+ \text{tr}(\pmb\Sigma^2) - 2\text{tr}(\mathbb E(\hat{\pmb\Sigma})\pmb\Sigma)\\
& = p \mathbb E\bigg[\int x^2 \mathrm{d}F^{\hat{\pmb \Sigma}}(x)\bigg] + \text{tr}(\pmb\Sigma^2)- 2\text{tr}(\pmb\Gamma\pmb\Sigma)\\
&= p\int x^2 \mathrm{d}F^{c_p, K_p}(x)+ p\mathbb E[L_{\pmb \Gamma}^{\hat{\pmb\Sigma}}(x^2)] + \text{tr}(\pmb\Sigma^2) \\ &\quad - 2\text{tr}(\pmb\Gamma\pmb\Sigma) \\
&= pc_p+ p\mathbb E[L_{\pmb \Gamma}^{\hat{\pmb\Sigma}}(x^2)] +\text{tr}(\pmb \Gamma^2) + \text{tr}(\pmb\Sigma^2) \\ &\quad- 2\text{tr}(\pmb\Gamma\pmb\Sigma) \\
&= pc_p+ p\mathbb E[L_{\pmb \Gamma}^{\hat{\pmb\Sigma}}(x^2)] + \|\pmb \Gamma - \pmb\Sigma\|_F^2.
\end{aligned}
\]
\end{proof}

\end{lemma}

The decomposition in Lemma~\ref{lem-bias-variance} is primarily a
bookkeeping device: it further separates the variance term in the
bias-variance decomposition of MSE into an irreducible contribution from
the Mar\v{c}enko-Pastur limit, plus a deviation,
\(\mathbb E(L_{\pmb \Gamma}^{\hat{\pmb\Sigma}}(x^2))\). Note that this
is an exact result; even for finite \(n,p\),
\(\mathbb E(L_{\pmb \Gamma}^{\hat{\pmb\Sigma}}(x^2))\) provides a way to
compare the pre-asymptotic risk of estimators that share a limiting
spectral distribution.

In special cases, this expectation can be derived analytically. In
general, however, the exact distribution of
\(L_{\pmb \Gamma}^{\hat{\pmb\Sigma}}(x^2)\) --- and thus its expectation
--- are unknown for an arbitrary estimator. For certain estimators,
central limit theorems for linear spectral statistics provide a way to
quantify how such estimators compare in the high-dimensional limit.

In particular, {[}15{]} derives a central limit theorem for the linear
spectral statistics of sample covariance matrices on certain
distributions. These were recently extended to a central limit theorem
for LSS of the SSCM by {[}6{]}, and to those of Tyler's M-estimator in
{[}13{]}. These papers employ a deterministic centering matrix,
specialized to (\ref{eq-model}) as \[
\mathbf T := \pmb\Sigma-\frac{2}{p}\bigg(\pmb\Sigma^2-\beta_{2,p}\pmb\Sigma\bigg),
\] to account for the nonlinearities induced by the denominators of the
SSCM and TME. In particular, {[}6{]} show that \(\mathbf T\)
approximates
\(\pmb\Gamma_{\text{SSCM}}:= \mathbb E(\hat{\pmb\Sigma}_{\mathrm{SSCM}})\)
to order \(o(p^{-1})\) in operator norm, which we denote
\(\|\cdot\|_{\text{op}}\) throughout. Finally, define the TME population
mean similarly as
\(\pmb\Gamma_{\text{TME}}:= \mathbb E(\hat{\pmb\Sigma}_{\mathrm{TME}})\).

For the spatial sign covariance matrix and Tyler's M-estimator, we can
quantify the LSS means of Lemma~\ref{lem-bias-variance} using the
following lemma.

\begin{lemma}[Naturally-Centered LSS Means for SSCM and Tyler's
M-Estimator]\protect\hypertarget{lem-errors}{}\label{lem-errors}

Suppose we have data generated by the model defined in (\ref{eq-model}).
Then, the following hold true.

\begin{enumerate}
\def\labelenumi{\alph{enumi})}
\tightlist
\item
  For the SSCM, for any \(n,p\),
  \(p\mathbb E(L_{\pmb \Gamma_{\text{SSCM}}}^{\hat{\pmb\Sigma}_\text{SSCM}}(x^2)) =-\frac{1}{n}\text{tr}(\pmb \Gamma_\text{SSCM}^2)\).
  When Assumption~\ref{exr-hd-regime} holds,
  \(p\mathbb E(L^{\hat{\pmb\Sigma}_\text{SSCM}}_{\pmb \Gamma_{\text{SSCM}}}(x^2)) \longrightarrow -c\beta_2\)
  as \(n,p \rightarrow \infty\);
\item
  Suppose Assumption~\ref{exr-hd-regime} holds with \(c_p, c<1\) so that
  Tyler's M-estimator exists. Then
  \(p\mathbb E(L_{\pmb \Gamma_{\text{TME}}}^{\hat{\pmb\Sigma}_\text{TME}}(x^2)) \longrightarrow c\beta_2+\frac{2c}{(1-c)}\)
  as \(n,p \rightarrow \infty\).
\end{enumerate}

\end{lemma}

\begin{proof}
For the SSCM, let \(\mathbf Y_i=\mathbf X_i/\|\mathbf X_i\|\) and
\(\pmb\Gamma_{\mathrm{SSCM}}=p\mathbb E(\mathbf Y_i\mathbf Y_i^\top)\).
Using \(\|\mathbf Y_i\|=1\) and independent observations, \[
\mathbb E\operatorname{tr}(\hat{\pmb\Sigma}_{\mathrm{SSCM}}^2)
=
pc_p+\left(1-\frac{1}{n}\right)
\operatorname{tr}(\pmb\Gamma_{\mathrm{SSCM}}^2).
\] Matching terms with Lemma~\ref{lem-bias-variance} therefore gives \[
p\mathbb E\!\left[
L_{\pmb\Gamma_{\mathrm{SSCM}}}^{\hat{\pmb\Sigma}_{\mathrm{SSCM}}}(x^2)
\right]
=
-\frac{1}{n}\operatorname{tr}(\pmb\Gamma_{\mathrm{SSCM}}^2)
 \underset{n,p \rightarrow\infty}{\longrightarrow} -c\beta_2,
\] where the limit follows from
\(\|\pmb\Gamma_{\mathrm{SSCM}}-\mathbf T\|_{\mathrm{op}}=o(p^{-1})\) and
\(\mathbf T=\pmb\Sigma+\mathcal O(p^{-1})\) under the bounded-spectrum
assumption {[}6{]}, so that

\[
\begin{aligned}
\operatorname{tr}(\pmb\Gamma_{\mathrm{SSCM}}^2) 
&=\operatorname{tr}(\mathbf T^2)+o(1)\\
&=\operatorname{tr}(\pmb\Sigma^2)+\frac{4}{p}\left[\beta_{2,p}\operatorname{tr}(\pmb\Sigma^2)-\operatorname{tr}(\pmb\Sigma^3)\right]+o(1)\\
&=\operatorname{tr}(\pmb\Sigma^2)+4(\beta_{2,p}^2-\beta_{3,p})+o(1).
\end{aligned}
\]

Observe that under this construction, we have a population-centered LSS
mean for the SSCM given by \[
\begin{aligned}
\mathbb E\operatorname{tr}(&\hat{\pmb\Sigma}_{\mathrm{SSCM}}^2)
-\operatorname{tr}(\pmb\Sigma^2)-pc_p \\
&=-\frac1n\operatorname{tr}(\pmb\Gamma_{\mathrm{SSCM}}^2) 
+\operatorname{tr}(\pmb\Gamma_{\mathrm{SSCM}}^2) -\operatorname{tr}(\pmb\Sigma^2)\\
&\underset{n,p \rightarrow\infty}{\longrightarrow}  -c\beta_2+4(\beta_2^2-\beta_3).
\end{aligned}
\]

For TME, Proposition~\ref{prp-tme-l2} implies that the limiting value of
\(\Pi_n\) given by ({[}13{]}, S5.3) holds in expectation:

\[
\mathbb E\operatorname{tr}(\hat{\pmb\Sigma}_{\mathrm{TME}}^2)-\mathbb E\operatorname{tr}(\hat{\pmb\Sigma}_{\mathrm{SSCM}}^2) \underset{n,p \rightarrow\infty}{\longrightarrow}  \frac{2c}{1-c}+2c\beta_2+4(\beta_3-\beta_2^2).
\] Adding this to the population-centered LSS mean for the SSCM, we get
\[
\mathbb E\operatorname{tr}(\hat{\pmb\Sigma}_{\mathrm{TME}}^2) - \text{tr}(\pmb\Sigma^2) -pc_p \underset{n,p \rightarrow\infty}{\longrightarrow} c\beta_2 + \frac{2c}{1-c}.
\]

Then by Lemma~\ref{lem-bias-variance},
\(p\mathbb E[L_{\pmb \Gamma_\text{TME}}^{\hat{\pmb\Sigma}_\text{TME}}(x^2)]=\mathbb E \text{tr}(\hat{\pmb\Sigma}_\text{TME}^2)-pc_p-\text{tr}(\pmb \Gamma_\text{TME}^2)\).
Since Proposition~\ref{prp-estimator-control} gives
\(\|\mathbb E (\hat{\pmb\Sigma}_{\text{TME}}) - \pmb\Sigma\|_F^2 = \mathcal O(1/p)\)
and
\(\text{tr}(\mathbb E (\hat{\pmb\Sigma}_{\text{TME}})\pmb\Sigma) = \text{tr}(\pmb\Sigma^2) + \mathcal O(1/p)\),
the result follows.
\end{proof}

The following result shows that, in a finite \(n,p\) setting, the SSCM
can achieve lower risk than any unbiased shape estimator for a Gaussian
scale mixture.

\begin{theorem}[Exact Frobenius Risk Dominance of SSCM Over Unbiased
Estimators For
GSM]\protect\hypertarget{thm-frob-lss}{}\label{thm-frob-lss}

For a sample generated by (\ref{eq-model}) under
Assumption~\ref{exr-gsm}, let \(\hat{\pmb\Sigma}\) be a symmetric
matrix-valued estimator of \(\pmb\Sigma\) that is unbiased on \(\Phi\).
Further, assume that the Fisher information matrix for the shape matrix
exists. If
\(\|\pmb\Gamma_{\mathrm{SSCM}}-\pmb\Sigma\|_F^2 \leq \frac{1}{n}\operatorname{tr}(\pmb\Gamma_{\mathrm{SSCM}}^2)+\frac{1}{n}\operatorname{tr}(\pmb\Sigma^2) - \frac{2}{n}\text{tr}(\pmb\Sigma^4)\text{tr}(\pmb\Sigma^2)^{-1}\),
then
\(\mathbb E\|\hat{\pmb\Sigma}_{\text{SSCM}} - \pmb\Sigma\|_F^2 \leq \mathbb E\|\hat{\pmb\Sigma} - \pmb\Sigma\|_F^2\).
The SSCM's risk advantage over any such estimator is at least the margin
by which this condition is satisfied.

\end{theorem}

The proof of Theorem~\ref{thm-frob-lss}, provided in the
\hyperref[sec-proof-appendix]{Appendix}, uses a Cramér--Rao bound under
the trace constraint. This bound involves the Fisher information matrix,
defined as the expected outer product of the gradient of the sample
log-likelihood with respect to the shape parameters. Intuitively,
Theorem~\ref{thm-frob-lss} shows that the SSCM's variance advantage can
compensate for the squared bias it incurs when \(\pmb\Sigma\) is not
equal to \(\mathbf I_p\) (where the SSCM is unbiased). Although the
theorem's condition involves unknown population quantities, under
Assumption~\ref{exr-hd-regime} its left-hand side is
\(\mathcal O(p^{-1})\), while its right-hand side is \(\mathcal O(1)\).
Practically speaking, this condition is satisfied fairly rapidly in
moderate dimensional problems. An immediate consequence is that the SSCM
asymptotically dominates every unbiased estimator irrespective of the
true shape matrix.

\begin{corollary}[Limiting Frobenius Risk Dominance of SSCM Over
Unbiased Estimators For
GSM]\protect\hypertarget{cor-frob-lss}{}\label{cor-frob-lss}

Consider a sample generated by (\ref{eq-model}), and suppose Assumptions
\ref{exr-hd-regime} and \ref{exr-gsm} hold. Let \(\hat{\pmb\Sigma}\) be
a symmetric matrix-valued estimator of \(\pmb\Sigma\) that is unbiased
on \(\Phi\). Further, suppose that the Fisher information matrix for the
shape matrix exists. Then \(\hat{\pmb\Sigma}_{\text{SSCM}}\)
asymptotically dominates every unbiased estimator \(\hat{\pmb\Sigma}\).

\end{corollary}

\begin{proof}
This follows from Theorem~\ref{thm-frob-lss}; the conclusion of the
theorem gives the following lower bound on the risk advantage:

\[
\begin{aligned}
\mathbb E\|\hat{\pmb\Sigma}  &- \pmb\Sigma\|_F^2 - \mathbb E\|\hat{\pmb\Sigma}_{\text{SSCM}} - \pmb\Sigma\|_F^2 \\
&\geq \frac{1}{n}\text{tr}(\pmb\Sigma^2) + \frac{1}{n}\text{tr}(\pmb\Gamma^2_\text{SSCM}) -
\frac{2}{n}
\frac{\operatorname{tr}(\pmb\Sigma^4)}
{\operatorname{tr}(\pmb\Sigma^2)} \\
& \qquad - \|\pmb\Gamma_\text{SSCM}-\pmb\Sigma\|_F^2.
\end{aligned}
\] By definition, \(p^{-1}\text{tr}(\pmb\Sigma^2)\rightarrow \beta_2\)
and
\(p^{-1}\text{tr}(\pmb\Gamma_\text{SSCM}^2) = p^{-1}\text{tr}(\pmb\Sigma^2) + o(1)\rightarrow \beta_2\),
while the remaining terms are \(o(1)\) by the bounded spectrum
assumption and Proposition~\ref{prp-estimator-control}. Therefore, the
SSCM risk advantage is bounded by \[
\liminf_{n,p \rightarrow \infty}\bigg[\mathbb E\|\hat{\pmb\Sigma}  - \pmb\Sigma\|_F^2 - \mathbb E\|\hat{\pmb\Sigma}_{\text{SSCM}} - \pmb\Sigma\|_F^2\bigg] \geq 2c\beta_2 > 0.
\]

Hence, in the high-dimensional limit, the SSCM has strictly smaller
Frobenius risk than every unbiased estimator \(\hat{\pmb\Sigma}\). Note
that the limit here is a \(\liminf\), since no convergence assumption is
made on the risk of an arbitrary unbiased estimator.
\end{proof}

The classical appeal of MLEs stems in large part from their unbiasedness
and efficiency under standard fixed-dimensional asymptotics. In the
high-dimensional regime considered here, the SSCM's squared bias
likewise vanishes, while its regularization delivers a persistent
variance advantage over every unbiased estimator. Compared with MLEs
based on a fully specified likelihood function for \(\mathbf X_i\), the
SSCM also avoids the associated risk of misspecification of the mixing
law.

Next, we show that this SSCM outperformance carries over to Tyler's
M-estimator with even looser model assumptions.

\begin{theorem}[High-Dimensional Frobenius Risk Dominance of SSCM Over
Tyler's
M-Estimator]\protect\hypertarget{thm-tme-dominance}{}\label{thm-tme-dominance}

Consider a sample generated by (\ref{eq-model}), and suppose
Assumption~\ref{exr-hd-regime} holds, with \(c< 1\). Then
\(\hat{\pmb\Sigma}_{\text{SSCM}}\) asymptotically dominates
\(\hat{\pmb\Sigma}_{\text{TME}}\).

\begin{proof}
\[
\begin{aligned}
\mathbb E\|\hat{\pmb\Sigma}_{\text{TME}} & - \pmb\Sigma\|_F^2 - \mathbb E\|\hat{\pmb\Sigma}_{\text{SSCM}} - \pmb\Sigma\|_F^2   \\
&=pc_p +  p\mathbb E[L^{\hat{\pmb\Sigma}_\text{TME}}_{\pmb \Gamma_{\text{TME}}}(x^2)] + \|\pmb\Gamma_{\text{TME}} - \pmb\Sigma\|_F^2 \\
&\quad - \bigg( pc_p + p\mathbb E[L^{\hat{\pmb\Sigma}_\text{SSCM}}_{\pmb \Gamma_{\text{SSCM}}}(x^2)] + \|\pmb\Gamma_{\text{SSCM}} - \pmb\Sigma\|_F^2\bigg) \\
&= p\mathbb E[L^{\hat{\pmb\Sigma}_\text{TME}}_{\pmb \Gamma_{\text{TME}}}(x^2)] - p\mathbb E[L^{\hat{\pmb\Sigma}_\text{SSCM}}_{\pmb \Gamma_{\text{SSCM}}}(x^2)] + \mathcal O(1/p)\\
&\underset{n,p \rightarrow \infty}{\longrightarrow} c\beta_2+\frac{2c}{(1-c)} + c\beta_2 \\
&= 2c\beta_2+\frac{2c}{(1-c)}>0.
\end{aligned}
\]

In the above, the risk as a function of LSS means is due to
Lemma~\ref{lem-bias-variance}. The Frobenius bias for both SSCM and TME
is vanishing as a result of Proposition~\ref{prp-estimator-control}, and
the limiting values of the centered LSS means are provided in
Lemma~\ref{lem-errors}.
\end{proof}

\end{theorem}

Theorem~\ref{thm-tme-dominance} applies to shape matrix estimation in
the more general model of (\ref{eq-model}), not merely the special case
of Gaussian scale mixtures required by Assumption~\ref{exr-gsm}.

\section{SIMULATION RESULTS}\label{simulation-results}

In this section, we provide numerical simulation results that
demonstrate the accuracy of the estimation errors predicted by
Lemma~\ref{lem-bias-variance}, using the limiting values of the
naturally-centered means of LSS provided in Lemma~\ref{lem-errors}, even
when the dimension is only moderately large. In doing so, we
intentionally leave out the finite-\(p\) bias-correction from the SSCM,
emphasizing its small penalty even in problems far from the asymptotic
regime.

Because the SSCM and TME are both distribution-free in the sense that
each estimator is invariant to the choice of the mixing variable \(W_i\)
in (\ref{eq-model}), we set \(W_i=1\) for \(i=1,\dots, n\) and simulate
from a multivariate Gaussian with mean zero and a trace-\(p\) covariance
matrix parameterized as follows: the dimension \(p\), a spike fraction
\(\epsilon\), and a spike magnitude \(\rho>1\). For each scenario, we
form a diagonal covariance matrix with \(\lceil \epsilon p \rceil\)
eigenvalues equal to \(\rho\) and the remaining eigenvalues equal to 1,
followed by normalization to enforce \(\text{tr}(\pmb\Sigma)=p\). Within
each scenario, every estimator is computed from the same 150 simulated
samples, so comparisons between estimators are paired: risk differences
are averaged replicate by replicate, and their Monte Carlo standard
errors account for the sampling variation the estimators share. Standard
errors for simulated risk are given in the
\hyperref[sec-supplement]{Supplement}.

In addition to the SSCM and TME, we also benchmark the sample covariance
matrix on this Gaussian data, which can be thought of as simulating the
unobservable latent Gaussian vector in (\ref{eq-model}). This is not a
perfectly matched benchmark for shape estimation, as the sample
covariance must also estimate an overall scale while TME and the SSCM
are trace-\(p\) by construction. Nevertheless, we use the unnormalized
sample covariance as a benchmark because it is the canonical covariance
estimator in the Gaussian setting, and provides a familiar reference
point for covariance estimation. Additionally, because it exploits the
Gaussian structure of the latent variables in (\ref{eq-model}), it can
be thought of as an ``oracle'' benchmark. Although the additional scale
estimation problem may appear to give the sample covariance a
disadvantage, this distinction becomes asymptotically negligible in the
high-dimensional regime considered here. We scale Frobenius errors by
\(1/p\) to make cleaner comparisons between scenarios of different
(finite) dimensions.

Fig.~\ref{fig-shape-sims} shows the limiting theoretical and average
simulated Frobenius risk over 150 draws as \(c_p=p/n\) varies, while
Fig.~\ref{fig-shape-sims-vs-p} depicts the same general simulation setup
but with \(c_p=0.5\) fixed and varying the dimension \(p\). These plots
show that even at a relatively low \(n,p\), simulated errors coincide
with their high-dimensional limiting values remarkably well. The
separation between estimators is far larger than simulation variability:
at every plotted point in Fig.~\ref{fig-shape-sims} and
Fig.~\ref{fig-shape-sims-vs-p}, the paired difference in risk between
the SSCM and each benchmark exceeds 12 Monte Carlo standard errors. We
provide a tabular summary of additional simulation scenarios in the
online \hyperref[sec-supplement]{Supplement} to this paper.

\begin{figure}[tbp]

\centering{

\includegraphics{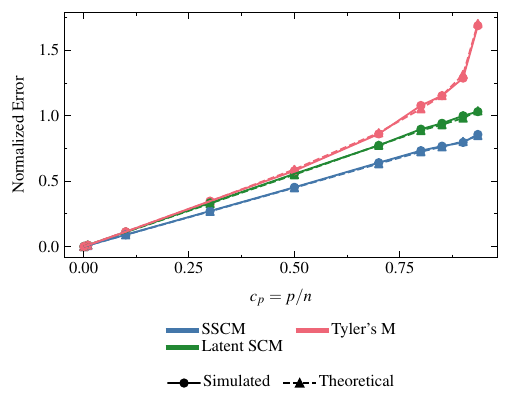}

}

\caption{\label{fig-shape-sims}Simulated vs.~high-dimensional
theoretical normalized error,
\(p^{-1}\mathbb E|\hat{\pmb\Sigma}-\pmb\Sigma|_F^2\), for \(p=50\),
\(\rho=10{,}000\), \(\epsilon=0.2\), and various values of \(c_p\).
Simulated values average 150 replicates per configuration. Latent SCM
denotes the latent Gaussian oracle sample covariance matrix used as a
benchmark.}

\end{figure}%

\begin{figure}[tbp]

\centering{

\includegraphics{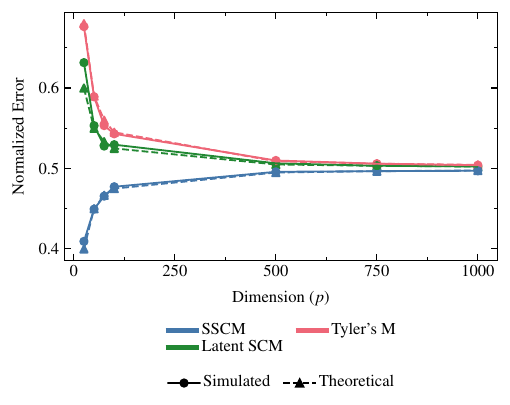}

}

\caption{\label{fig-shape-sims-vs-p}Simulated vs.~theoretical normalized
error, \(p^{-1}\mathbb E|\hat{\pmb\Sigma}-\pmb\Sigma|_F^2\), for \(c_p\)
= 0.5, \(\rho\) = 10,000, \(\epsilon\) = 0.2, and various values of
\(p\). The plots above show the average risk over 150 simulations for
each configuration.}

\end{figure}%

As one would expect from the fact that the sample covariance matrix,
SSCM, and TME share a limiting spectral distribution,
Fig.~\ref{fig-shape-sims-vs-p} demonstrates that the normalized risk of
shape matrix estimates converges to the same value as
\(n, p \rightarrow \infty\). Importantly, however, this figure clearly
shows that as a consequence of its implicit regularization, the SSCM
converges from below. Likewise, the relative benefits of SSCM's
regularization are the greatest in moderate dimensions, where the LSS
error term is not yet dwarfed by the total risk of the estimator (i.e.,
the other terms in (\ref{eq-risk-decomp})). For example, at
\((p,n,\epsilon, \rho)=(25, 250, 0.2, 100)\), the SSCM provides a 29.1\%
(SE 1.6\%) and 30.2\% (SE 1.2\%) improvement over the latent sample
covariance and TME respectively, with an order-of-magnitude runtime
reduction over TME.

\begin{figure}[tbp]

\centering{

\includegraphics{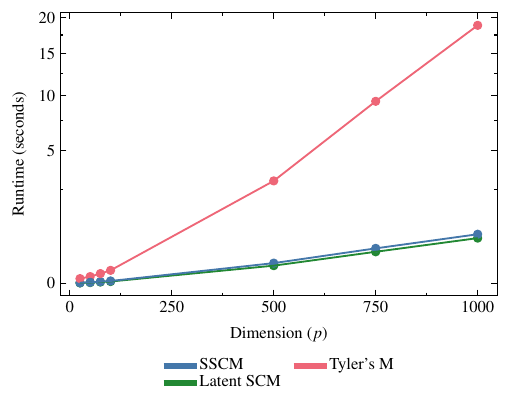}

}

\caption{\label{fig-shape-runtime}Estimator runtime comparisons for
\(\rho\) = 10,000, \(\epsilon\) = 0.2 at various dimensions, averaged
over values of \(c_p\).}

\end{figure}%

Conversely, when the dimension is large enough that the differences are
not practically significant, the computational benefits of the SSCM
versus more complex algorithms such as TME are material. Example
computational costs for each estimator in problems of increasing
dimensionality are shown in Fig.~\ref{fig-shape-runtime}.

\section{CONCLUSION}\label{conclusion}

Using the criterion in Definition~\ref{def-asymptotic-dominance}, we
show that the SSCM asymptotically dominates every unbiased shape
estimator for Gaussian scale mixtures, as well as Tyler's M-estimator
under the broader model in (\ref{eq-model}). These improvements in
Frobenius risk arise from implicit regularization of the SSCM in the
high-dimensional regime: its variance reduction persists while the
squared bias vanishes as dimension grows.

Simulation studies show that the theoretical limiting risk is remarkably
accurate even when the ambient dimension of the data is only moderate.
Moreover, the comparatively cheap computational cost of the SSCM in high
dimensions makes it an appealing choice even in situations where the
theoretical performance difference is insignificant in practical terms.

While Corollary~\ref{cor-frob-lss} shows that the Spatial Sign
Covariance Matrix dominates all unbiased estimators in the
high-dimensional limit, we must stress that this does not imply
dominance over all \emph{asymptotically} unbiased estimators such as
MLEs, which may be biased in finite samples and can attain better
efficiency than the Cramér-Rao lower bound allows without violating any
theoretical laws. Additionally, unbiasedness under fixed-\(p\)
asymptotics does not guarantee unbiasedness in the
\(n,p \rightarrow \infty\) regime we study. There is currently no
theoretical framework to establish the general dominance of the SSCM
over MLEs in the high-dimensional limit. That said, improving upon MLE
performance for heavy-tailed elliptical vectors was a primary motivation
for this research, and empirical results against the benchmarks
presented in the previous section show that the SSCM can and does
outperform MLEs in the settings studied here. These results suggest that
the SSCM merits serious consideration by practitioners as an alternative
to likelihood-based shape estimators.

\section*{\texorpdfstring{APPENDIX: PROOF OF THEOREM
\ref{thm-frob-lss}}{APPENDIX: PROOF OF THEOREM }}\label{sec-proof-appendix}
\addcontentsline{toc}{section}{APPENDIX: PROOF OF THEOREM
\ref{thm-frob-lss}}

\begin{proof}
The proof proceeds in three steps. First, we show that the latent
Gaussian model yields a Cramér-Rao bound no larger than that of any
corresponding Gaussian scale mixture. We then derive this bound for the
shape matrix and finally compare the resulting risk to that of the SSCM.
Because the CRB bounds the risk of any unbiased estimator, any
conditions that result in SSCM risk lower than this benchmark risk also
imply that SSCM dominates all unbiased estimators of \(\pmb\Sigma\).

Let \(\hat{\pmb\Sigma}\) be a symmetric matrix-valued estimator of
\(\pmb\Sigma\) that is unbiased on \(\Phi\). Recall that under
(\ref{eq-model}), \[
\pmb G_i\sim N_p(\mathbf 0,\pmb\Sigma)\text{ and } \operatorname{tr}(\pmb\Sigma)=p,
\] so the shape matrix of interest is also exactly the covariance matrix
of the latent Gaussian vector, which is independent of \(W_i\) by
Assumption~\ref{exr-gsm}.

To obtain the Gaussian Cramér--Rao bound for the trace-\(p\) shape
matrix, we start with a standard unrestricted covariance
parameterization in the Gaussian model and apply the
equality-constrained Cramér--Rao bound of {[}16{]}. In addition to the
trace-\(p\) parameter space \(\Phi\) defined in the problem formulation,
we introduce the ambient unrestricted parameter space represented in
(nonredundant) half-vectorization coordinates: \[
\Theta
:=
\left\{
\operatorname{vech}(\pmb\Sigma):
\pmb\Sigma\in\mathbb S_{++}^p
\right\}.
\] We write \(\boldsymbol\theta\in\Theta\) for the unrestricted
parameter vector and \(\boldsymbol\phi\in\Phi\) for its
trace-constrained counterpart. It is convenient to represent the
unrestricted parameter in the nonredundant \(\operatorname{vech}\)
coordinates in \(\mathbb R^{p(p+1)/2}\), while expressing the
constrained shape parameter in vec coordinates in \(\mathbb R^{p^2}\);
to bridge these representations, let \(\mathbf D_p\) denote the
duplication matrix, so that

\[
\operatorname{vec}(\pmb\Sigma)
=
\mathbf D_p\boldsymbol\theta.
\]

The trace constraint can be written as \[
a(\boldsymbol\theta)
=
\operatorname{tr}(\pmb\Sigma)-p
=
\operatorname{vec}(\mathbf I_p)^\top
\mathbf D_p\boldsymbol\theta-p
=
0,
\] and its gradient is therefore \[
\nabla_{\boldsymbol\theta}a(\boldsymbol\theta)
=
\mathbf D_p^\top\operatorname{vec}(\mathbf I_p).
\] By {[}16{]}, there exists \(\mathbf U\) with orthonormal columns
spanning the nullspace of the constraint gradient, so that \[
\nabla_{\boldsymbol\theta}a(\boldsymbol\theta)^\top\mathbf U
=
\mathbf 0
\text{ and }
\mathbf U^\top\mathbf U
=
\mathbf I.
\]

Let \(\mathcal I_{\mathbf X}\) and \(\mathcal I_{\mathbf G}\) denote the
Fisher Information Matrices (FIMs) with respect to \(\boldsymbol\theta\)
based on the observed Gaussian scale mixture and latent Gaussian sample,
respectively. Since \(W_i\) itself carries no information about
\(\pmb\Sigma\), while \(\pmb G_i=\pmb X_i/\sqrt{W_i}\) is recovered from
the complete data \((\pmb X_i,W_i)\), the missing-information principle
{[}17{]} gives \[
\mathcal I_{\mathbf G}
=
\mathcal I_{\mathbf X}
+
\mathbb E_{\mathbf X}\!\left[
\operatorname{Var}\!\left(
\nabla_{\boldsymbol\theta}
\log f_{\mathbf X\mid W}
(\mathbf X\mid W,\boldsymbol\theta)
\mid
\mathbf X
\right)
\right]
\succeq
\mathcal I_{\mathbf X}.
\]

By Theorem 1 of {[}16{]}, the equality-constrained CRB matrices for the
vectorized shape parameter are \[
\operatorname{CRB}_{\pmb\phi,\mathbf G}
=
\mathbf D_p\mathbf U
\left(
\mathbf U^\top
\mathcal I_{\mathbf G}
\mathbf U
\right)^{-1}
\mathbf U^\top\mathbf D_p^\top, \text{ and}
\] \[
\operatorname{CRB}_{\pmb\phi,\mathbf X}
=
\mathbf D_p\mathbf U
\left(
\mathbf U^\top
\mathcal I_{\mathbf X}
\mathbf U
\right)^{-1}
\mathbf U^\top\mathbf D_p^\top.
\] This implies that \[
\mathbf U^\top
\mathcal I_{\mathbf G}
\mathbf U
\succeq
\mathbf U^\top
\mathcal I_{\mathbf X}
\mathbf U,
\] where inversion reverses the positive-semidefinite ordering, and
congruence by \(\mathbf D_p\mathbf U\) preserves it. Hence \[
\operatorname{CRB}_{\pmb\phi,\mathbf G}
\preceq
\operatorname{CRB}_{\pmb\phi,\mathbf X}.
\]

Finally, we obtain the relevant Gaussian bound explicitly. For \(n\)
independent Gaussian observations with covariance \(\pmb\Sigma\), the
FIM under the nonredundant parameterization
\(\boldsymbol\theta=\operatorname{vech}(\pmb\Sigma)\) is \[
\mathcal I_{\mathbf G}
=
\frac{n}{2}
\mathbf D_p^\top
(\pmb\Sigma^{-1}\otimes\pmb\Sigma^{-1})
\mathbf D_p,
\] which is nonsingular {[}18{]}. Thus, defining the unrestricted
Gaussian covariance bound by \[
\operatorname{CRB}_{\boldsymbol\theta,\mathbf G}
=
\mathcal I_{\mathbf G}^{-1},
\] Corollary 1 of {[}16{]} and the transformation from
\(\boldsymbol\theta\) to \(\boldsymbol\phi\) give
\begin{equation}\phantomsection\label{eq-crb-constrained}{
\begin{aligned}
\operatorname{CRB}_{\pmb\phi,\mathbf G}
&=
\mathbf D_p\Bigg(
\operatorname{CRB}_{\boldsymbol\theta,\mathbf G}
-
\operatorname{CRB}_{\boldsymbol\theta,\mathbf G}
\nabla_{\boldsymbol\theta}a(\boldsymbol\theta) \\
&\qquad\times
\Big[
\nabla_{\boldsymbol\theta}a(\boldsymbol\theta)^\top
\operatorname{CRB}_{\boldsymbol\theta,\mathbf G}
\nabla_{\boldsymbol\theta}a(\boldsymbol\theta)
\Big]^{-1} \\
&\qquad\times
\nabla_{\boldsymbol\theta}a(\boldsymbol\theta)^\top
\operatorname{CRB}_{\boldsymbol\theta,\mathbf G}
\Bigg)\mathbf D_p^\top.
\end{aligned}
}\end{equation} The first term is simply the unrestricted CRB matrix
represented in \(\operatorname{vec}\) coordinates: \[
\mathbf D_p
\operatorname{CRB}_{\boldsymbol\theta,\mathbf G}
\mathbf D_p^\top
=
\frac{1}{n}
(\mathbf I_{p^2}+\pmb{\mathcal K})
(\pmb\Sigma\otimes\pmb\Sigma),
\] where \(\pmb{\mathcal K}\) is the commutation matrix {[}18{]}. Using
\(\nabla_{\boldsymbol\theta}a(\boldsymbol\theta)
=
\mathbf D_p^\top\operatorname{vec}(\mathbf I_p)\) then gives \[
\mathbf D_p
\operatorname{CRB}_{\boldsymbol\theta,\mathbf G}
\nabla_{\boldsymbol\theta}a(\boldsymbol\theta)
=
\frac{2}{n}\operatorname{vec}(\pmb\Sigma^2), \text{ and }
\] \[
\nabla_{\boldsymbol\theta}a(\boldsymbol\theta)^\top
\operatorname{CRB}_{\boldsymbol\theta,\mathbf G}
\nabla_{\boldsymbol\theta}a(\boldsymbol\theta)
=
\frac{2}{n}\operatorname{tr}(\pmb\Sigma^2).
\] Plugging these into (\ref{eq-crb-constrained}), we arrive at \[
\operatorname{CRB}_{\pmb\phi,\mathbf G}
=
\frac{1}{n}
(\mathbf I_{p^2}+\pmb{\mathcal K})
(\pmb\Sigma\otimes\pmb\Sigma)
-
\frac{2}{n\operatorname{tr}(\pmb\Sigma^2)}
\operatorname{vec}(\pmb\Sigma^2)
\operatorname{vec}(\pmb\Sigma^2)^\top.
\]

For any unbiased estimator based on Gaussian scale mixture observations,
the Cramér--Rao bound and the ordering above imply \[
\operatorname{Cov}\!\left(
\operatorname{vec}(\hat{\pmb\Sigma})
\right)
\succeq
\operatorname{CRB}_{\pmb\phi,\mathbf X}
\succeq
\operatorname{CRB}_{\pmb\phi,\mathbf G}.
\] Additionally, because \(\hat{\pmb\Sigma}\) is unbiased, the Frobenius
risk equals the total variance of its vectorized entries. Taking traces
therefore gives the risk bound for unbiased estimators as \[
\begin{aligned}
\mathbb E\|\hat{\pmb\Sigma}-\pmb\Sigma\|_F^2
&=
\operatorname{tr}\!\left[
\operatorname{Cov}\!\left(
\operatorname{vec}(\hat{\pmb\Sigma})
\right)
\right] \\
&\geq
\operatorname{tr}\!\left(
\operatorname{CRB}_{\pmb\phi,\mathbf G}
\right) 
\\
&= pc_p
+
\frac{1}{n}\operatorname{tr}(\pmb\Sigma^2)
-
\frac{2}{n}
\frac{\operatorname{tr}(\pmb\Sigma^4)}
{\operatorname{tr}(\pmb\Sigma^2)}.
\end{aligned}
\] Next, let
\(\pmb\Gamma_{\mathrm{SSCM}}=\mathbb E(\hat{\pmb\Sigma}_{\mathrm{SSCM}})\).
By Lemma~\ref{lem-bias-variance} and Lemma~\ref{lem-errors}, the SSCM
has exact Frobenius risk given by \[
\mathbb E\|\hat{\pmb\Sigma}_{\mathrm{SSCM}}-\pmb\Sigma\|_F^2
=
pc_p
-\frac{1}{n}\operatorname{tr}(\pmb\Gamma_{\mathrm{SSCM}}^2)
+\|\pmb\Gamma_{\mathrm{SSCM}}-\pmb\Sigma\|_F^2.
\] Subtracting this expression from the preceding lower bound gives the
quantitative risk advantage stated in the theorem. Therefore, the SSCM
has no greater risk than the Gaussian Cramér--Rao lower bound whenever
\[
\begin{aligned}
&pc_p
-\frac{1}{n}\operatorname{tr}(\pmb\Gamma_{\mathrm{SSCM}}^2)
+\|\pmb\Gamma_{\mathrm{SSCM}}-\pmb\Sigma\|_F^2
\leq \\
&\qquad
pc_p
+
\frac{1}{n}\operatorname{tr}(\pmb\Sigma^2)
-
\frac{2}{n}
\frac{\operatorname{tr}(\pmb\Sigma^4)}
{\operatorname{tr}(\pmb\Sigma^2)}.
\end{aligned}
\] Rearranging yields \[
\begin{aligned}
\|\pmb\Gamma_{\mathrm{SSCM}}-&\pmb\Sigma\|_F^2
\leq \\ 
&\frac{1}{n}\operatorname{tr}(\pmb\Gamma_{\mathrm{SSCM}}^2)
+\frac{1}{n}\operatorname{tr}(\pmb\Sigma^2)
-
\frac{2}{n}
\frac{\operatorname{tr}(\pmb\Sigma^4)}
{\operatorname{tr}(\pmb\Sigma^2)},
\end{aligned}
\] which is exactly the condition stated in Theorem~\ref{thm-frob-lss}.
\end{proof}

\section*{ACKNOWLEDGMENT}\label{acknowledgment}
\addcontentsline{toc}{section}{ACKNOWLEDGMENT}

We thank Marco Avella Medina and John P. Nolan for their insightful
conversations and literature pointers.

AI Disclosure: The authors used AI tools, including Refine, OpenAI's
GPT-6 and Anthropic's Fable 5.1, to assist with formatting, literature
searches, technical review, and polishing prose. AI was not used for
proof generation. The authors retain full responsibility for the content
and results.

\section*{REFERENCES}\label{references}
\addcontentsline{toc}{section}{REFERENCES}

\phantomsection\label{refs}

%
%
\onecolumn
\setcounter{section}{0}
\renewcommand{\thesection}{S\arabic{section}}
\setcounter{equation}{0}
\renewcommand{\theequation}{S\arabic{equation}}

\begin{center}
\vskip0.25in{\LARGE\bfseries Supplement\par}\vskip1.0em
\phantomsection\label{sec-supplement}
\end{center}

This online supplement collects the proofs of
Proposition~\ref{prp-tme-l2} and
Proposition~\ref{prp-estimator-control}, as well as tables of simulation
results.

\section{\texorpdfstring{PROOF OF
Proposition~\ref{prp-tme-l2}}{PROOF OF Proposition~}}\label{proof-of-prp-tme-l2}

The machinery needed to prove the uniform integrability of \(\{\Pi_n\}\)
is provided in the Supplement to {[}13{]}; where possible, we introduce
new objects here borrowing their notation. Throughout, the standing
conditions of (\ref{eq-model}) and Assumption~\ref{exr-hd-regime} hold
with \(c_p, c\in(0,1)\). Finally, recall that
\(\operatorname{tr}(\hat{\pmb\Sigma}_{\mathrm{TME}}^2)\) and
\(\operatorname{tr}(\hat{\pmb\Sigma}_{\mathrm{SSCM}}^2)\) are understood
to belong to sequences indexed by \(n\) and \(p\), with indices
suppressed for readability.

We begin with the construction used to prove Lemma S1.2 in {[}13{]}.
Whenever all eigenvalues of both \(\hat{\pmb\Sigma}_{\mathrm{TME}}\) and
\(\hat{\pmb\Sigma}_{\mathrm{SSCM}}\) fall inside a fixed contour
\(\mathcal C\), Cauchy's integral formula gives

\begin{equation}\phantomsection\label{eq-cauchy-integral}{
\Pi_n = \operatorname{tr}(\hat{\pmb\Sigma}_{\mathrm{TME}}^2)-\operatorname{tr}(\hat{\pmb\Sigma}_{\mathrm{SSCM}}^2)= -\frac{1}{2\pi i} \oint_{\mathcal C}z^2 M_n(z)dz,
}\end{equation} where
\(M_n(z) = \operatorname{tr}[(\hat{\pmb\Sigma}_{\mathrm{TME}}-z\mathbf I_p)^{-1}] - \operatorname{tr}[(\hat{\pmb\Sigma}_{\mathrm{SSCM}}-z\mathbf I_p)^{-1}]\)
is the difference in the scaled Stieltjes transforms of TME and the
SSCM. Here \(z\in\mathbb C\), and we choose \(\mathcal C\) to be the
counterclockwise boundary of a rectangle with vertical sides at
\(x_l<x_r\) and horizontal sides at \(\operatorname{Im}z=\pm v_0\),
where \(v_0>0\).

The representation in (\ref{eq-cauchy-integral}) cannot be used without
modification, unfortunately. As \(n,p\) grow, realized eigenvalues of
both \(\hat{\pmb\Sigma}_{\mathrm{TME}}\) and
\(\hat{\pmb\Sigma}_{\mathrm{SSCM}}\) may fall outside any chosen fixed
contour or approach its real-axis crossings. To control these
situations, we use the regularized transform \(\widehat M_n\) of
{[}13{]}. The authors' results additionally require bounded standardized
coordinates, described below. We introduce one remainder, \(r_n^{(A)}\),
for the contour approximation and another, \(r_n^{(B)}\), for the
coordinate truncation, and show that both modifications to \(\Pi_n\) in
(\ref{eq-cauchy-integral}) vanish in \(L^2\).

Throughout, we assume the bounded-coordinate samples are used in every
object except \(\Pi_n\). Likewise, let \(K\) denote a positive constant,
independent of \(n\) and \(p\) (and of \(z\in\mathcal C\)), whose value
may change from line to line. Define \(\Pi_n^B\) as the version of
\(\Pi_n\) computed from the bounded-coordinate samples as described in
the section below, with the corresponding difference
\(r_n^{(B)} = \Pi_n - {\Pi}_n^B\).

Like {[}15{]} before them, {[}13{]} adopt a regularization of
\(M_n(z)\).

Let \(\delta_n=\varepsilon_n/n\), where \(\varepsilon_n\downarrow0\) and
\(\varepsilon_n>n^{-a}\) for some \(a\in(0,1)\), and define \[
\mathcal C_n=\{z\in\mathcal C:|\operatorname{Im}z|\ge\delta_n\}.
\] For \(z=x+iv\in\mathcal C\), we define the regularized transform
\(\widehat M_n\) as \[
\widehat M_n(x+iv)=
\begin{cases}
M_n(x+iv),&x+iv\in\mathcal C_n,\\
M_n(x+i\delta_n),&x\in\{x_l,x_r\},\quad 0\le v<\delta_n,\\
M_n(x-i\delta_n),&x\in\{x_l,x_r\},\quad -\delta_n<v<0.
\end{cases}
\] Thus \(\widehat M_n\) agrees with \(M_n\) on \(\mathcal C_n\) and
holds its value constant when \(z\) approaches the real axis.

Next, define

\[
\begin{aligned}
\widehat \Pi_n^B := -\frac{1}{2\pi i} \oint_{\mathcal C}z^2 \widehat M_n(z)dz, && \text{so that } r_n^{(A)} = \Pi_n^B - \widehat \Pi_n^B.
\end{aligned}
\] Then we have

\[ 
\begin{aligned}
\Pi_n & = \widehat{\Pi}_n^B + r_n^{(A)} + r_n^{(B)}.
\end{aligned}
\] To establish uniform integrability, we can show that
\(\sup_n \mathbb E|\Pi_n|^2 <\infty\). Observe that by Cauchy-Schwarz,

\[
|\Pi_n|^2 = |\widehat{\Pi}_n^B + r_n^{(A)} + r_n^{(B)}|^2  \leq 3(|\widehat{\Pi}_n^B|^2 + |r_n^{(A)}|^2 + |r_n^{(B)}|^2 ),
\]

thus, taking expectations in this inequality shows that a uniform
\(L^2\) bound for \(\widehat{\Pi}_n^B\) transfers to \(\Pi_n\) if both
remainders vanish in \(L^2\). We first show that
\(\mathbb E|r_n^{(A)}|^2\to0\) and \(\mathbb E|r_n^{(B)}|^2\to0\), then
establish the required uniform bound for \(\widehat{\Pi}_n^B\).

\subsection{\texorpdfstring{Control of Bounded Coordinates:
\(r_n^{(B)}\)}{Control of Bounded Coordinates: r\_n\^{}\{(B)\}}}\label{control-of-bounded-coordinates-r_nb}

To control higher moments in their proofs, {[}13{]} require the
coordinates of the latent random vectors allowed by their model to be
bounded. Under (\ref{eq-model}), the corresponding coordinates are
standard Gaussian and have moments of all orders. To apply their
intermediate bounds as stated, we show that truncating these coordinates
has a negligible effect on \(\Pi_n\) in \(L^2\).

Write \(\mathbf G=\mathbf Z\pmb\Sigma^{1/2}\). We retain each standard
Gaussian coordinate whose magnitude is at most \(n^{1/3}\) and
independently redraw any remaining coordinate from the standard Gaussian
law conditioned on this interval. The resulting bounded coordinates
satisfy Assumption 5 of {[}13{]}. Call the event
\(\mathcal A_n = \{\text{at least one entry of } \mathbf Z \text{ is replaced} \}\).

Standard Gaussian tail bounds plus a union bound over the \(np\) entries
of the matrix \(\mathbf Z\) give \[
P(\mathcal A_n)
\le 2npe^{-n^{2/3}/2}.
\] Since both \({\Pi}_n\) and \({\Pi}_n^B\) are based on positive
semidefinite estimators with trace \(p\), we have
\(|{\Pi}_n|,|{\Pi}_n^B|\le p^2\), and hence \[
\mathbb E |r_n^{(B)} |^2 = \mathbb E|{\Pi}_n-{\Pi}_n^B|^2
\le 4p^4 P(\mathcal A_n)
\le 8np^5e^{-n^{2/3}/2}\rightarrow 0.
\]

\subsection{\texorpdfstring{Control of the Regularization of \(M_n(z)\):
\(r_n^{(A)}\)}{Control of the Regularization of M\_n(z): r\_n\^{}\{(A)\}}}\label{control-of-the-regularization-of-m_nz-r_na}

Next, observe that \(\widehat M_n\) differs from \(M_n\) only on the
four segments
\(\{ x \pm iv: x\in \{x_l, x_r\}, 0 \leq v \leq n^{-1}\varepsilon_n\}\).
Thus the effect of regularization only modifies a segment of total
length \(4\varepsilon_n/n\) where \(\mathcal C\) crosses the real axis.

We begin by considering the event
\(\mathcal E_n = \{\text{all eigenvalues of }\hat{\pmb\Sigma}_{\mathrm{SSCM}}\text{ and }\hat{\pmb\Sigma}_{\mathrm{TME}} \text{ lie in }[\eta_l,\eta_r] \},\)
where \(x_l<\eta_l < \lambda_-(1-\sqrt{c})^2\) and
\(\lambda_+(1+\sqrt{c})^2< \eta_r < x_r\). By Lemmas S8.4 and S1.1,
\(P(\mathcal E_n) = 1-o(n^{-s})\) for every fixed \(s>0\). Call
\(d := \min(\eta_l - x_l, x_r - \eta_r) > 0\). Then because
\(\widehat M_n\), \(M_n\) each have \(2p\) terms, each with magnitude at
most \(1/d\) where they differ, \(|M_n|, |\widehat M_n| \leq 2p/d\), and
thus their difference is at most \(4p/d\).

Then on \(\mathcal E_n\), \[
|{r_n^{(A)}}| = |\Pi_n^B - \widehat \Pi_n^B | = \bigg|\frac{1}{2\pi i} \oint_{\mathcal C} z^2 [M_n(z) - \widehat M_n(z)]dz\bigg| \leq \frac{1}{2 \pi}\max_{\mathcal C}|z|^2\frac{4p}{d}\frac{4\varepsilon_n}{n}= K\frac{p\varepsilon_n}{n},
\] which gives

\[
\mathbb E (|{r_n^{(A)}}|^2\mathbf 1\{\mathcal E_n\}) \leq K^2\bigg(\frac{p\varepsilon_n}{n}\bigg)^2 \rightarrow 0.
\] Next, we evaluate the same for the complement event,
\(\mathcal E_n^\text{C}\). By construction, every evaluation of \(M_n\)
used to define \(\widehat M_n\) is at least \(\varepsilon_n/n\) from the
real axis where the eigenvalues occur, so
\(|\widehat M_n| \leq 2pn/\varepsilon_n\). Then

\[
\bigg|-\frac{1}{2\pi i} \oint_{\mathcal C}z^2 \widehat M_n(z)dz \bigg| = |\widehat \Pi_n^B| \leq \frac{1}{2\pi}\max_{\mathcal C}|z|^2\frac{2pn}{\varepsilon_n}\ell(\mathcal C) = K\frac{pn}{\varepsilon_n},
\] with \(\ell(\mathcal C)\) denoting the (finite) length of the
contour. Using the facts that \(\varepsilon_n > n^{-a}\) and
\(p = \mathcal O(n)\), and reusing the fact that \(|{\Pi}_n^B|\le p^2\),

\[
|{r_n^{(A)}}|= |\Pi_n^B - \widehat \Pi_n^B |\leq |\Pi_n^B| + |\widehat \Pi_n^B| \leq p^2 + K_1\frac{pn}{\varepsilon_n} \leq K_2n^{2+a}.
\] Then \[
\mathbb E (|{r_n^{(A)}}|^2\mathbf 1\{\mathcal E_n^\text{C}\}) \leq n^{4+2a}o(n^{-s}) \rightarrow 0,
\] for \(s > 6\). Finally, we have

\[
\mathbb E (|{r_n^{(A)}}|^2)  = \mathbb E (|{r_n^{(A)}}|^2\mathbf 1\{\mathcal E_n\}) + \mathbb E (|{r_n^{(A)}}|^2\mathbf 1\{\mathcal E_n^\text{C}\})  \rightarrow 0.
\]

Thus, \(\mathbb E |{r_n^{(A)}}|^2 \rightarrow 0\) and
\(\mathbb E |{r_n^{(B)}}|^2 \rightarrow 0\) so it remains to show a
uniform \(L^2\) bound for \(\widehat{\Pi}_n^B\).

\subsection{\texorpdfstring{Uniform \(L^2\) Bound for
\(\widehat{\Pi}_n^B\)}{Uniform L\^{}2 Bound for \textbackslash widehat\{\textbackslash Pi\}\_n\^{}B}}\label{uniform-l2-bound-for-widehatpi_nb}

Observe that since \(\oint_{\mathcal C} z^2 \widehat M_n(z_0)dz=0\) for
any fixed \(z_0 \in \mathcal C_n\), we can represent
\(\widehat{\Pi}_n^B\) as a shifted integral

\[
\begin{aligned}
\widehat{\Pi}_n^B &= -\frac{1}{2\pi i} \oint_{\mathcal C}z^2 \widehat M_n(z)dz \\
&= -\frac{1}{2\pi i} \oint_{\mathcal C}z^2 \widehat M_n(z)dz  + \frac{1}{2\pi i}\oint_{\mathcal C} z^2 \widehat M_n(z_0)dz \\
&= -\frac{1}{2\pi i} \oint_{\mathcal C}z^2 [\widehat M_n(z) - \widehat M_n(z_0)]dz
\end{aligned}
\]

Next, ({[}13{]}, S1.39) and the boundedness of the fixed contour
\(\mathcal C\) imply that there exists a constant \(K\) such that

\[
\begin{aligned}
\sup_{n;z, z_0 \in \mathcal{C}_n}\mathbb E|\widehat M_n(z) - \widehat M_n(z_0)|^2 \leq K
\end{aligned}
\] For each \(z\in\mathcal C\setminus\mathcal C_n\), the regularization
\(\widehat M_n(z)\) equals its value at a point in \(\mathcal C_n\).
Since \(\mathcal C\) is fixed and bounded, the squared distances between
all such evaluation points and \(z_0\in\mathcal C_n\) are uniformly
bounded. Enlarging \(K\) if necessary, we therefore obtain \[
\sup_{\substack{z\in\mathcal C\\ z_0\in\mathcal C_n}}
\mathbb E|\widehat M_n(z)-\widehat M_n(z_0)|^2
\le K,
\] uniformly for all sufficiently large \(n,p\).

Finally, we arrive at

\[
\begin{aligned}
\mathbb E|\widehat\Pi_n^B|^2
&=\frac{1}{4\pi^2}\mathbb E
\left|\oint_{\mathcal C}z^2
[\widehat M_n(z)-\widehat M_n(z_0)]\,dz\right|^2\\
&\le \frac{1}{4\pi^2}\mathbb E\left[
\left(\oint_{\mathcal C}|z|^2
|\widehat M_n(z)-\widehat M_n(z_0)|\,|dz|\right)^2
\right]\\
&\le \frac{\ell(\mathcal C)}{4\pi^2}
\mathbb E\left[
\oint_{\mathcal C}|z|^4
|\widehat M_n(z)-\widehat M_n(z_0)|^2\,|dz|
\right]\\
&=\frac{\ell(\mathcal C)}{4\pi^2}
\oint_{\mathcal C}|z|^4
\mathbb E|\widehat M_n(z)-\widehat M_n(z_0)|^2\,|dz|\\
&\le \frac{K \ell(\mathcal C)^2}{4\pi^2}
\max_{z\in\mathcal C}|z|^4.
\end{aligned}
\]

Above, the first inequality is the triangle inequality for contour
integrals. The second is Cauchy--Schwarz with respect to arc length,
applied to \(|z|^2|\widehat M_n(z)-\widehat M_n(z_0)|\) and the constant
function \(1\). Tonelli's theorem permits the interchange of expectation
and integration.

The final bound is constant and independent of \(n,p\). Together with
the decomposition \(\Pi_n=\widehat\Pi_n^B+r_n^{(A)}+r_n^{(B)}\) and the
vanishing second moments of both remainders, this yields
\(\sup_n\mathbb E|\Pi_n|^2<\infty\). Thus \(\{\Pi_n\}\) is uniformly
integrable.

\section{\texorpdfstring{PROOF OF
Proposition~\ref{prp-estimator-control}}{PROOF OF Proposition~}}\label{proof-of-prp-estimator-control}

Proposition~\ref{prp-estimator-control} is made up of three statements,
which comprise Lemma~\ref{lem-pop-sscm-risk} and
Lemma~\ref{lem-tme-control} below.

\begin{lemma}[Frobenius Bias for
SSCM]\protect\hypertarget{lem-pop-sscm-risk}{}\label{lem-pop-sscm-risk}

Suppose Assumption~\ref{exr-hd-regime} holds. Then
\(\|\pmb \Gamma_{\text{SSCM}} - \pmb\Sigma\|_F^2 = \mathcal O (1/p)\).

\end{lemma}

\begin{proof}
By Assumption~\ref{exr-hd-regime}, the eigenvalues \(\lambda_k\) of
\(\pmb\Sigma\) satisfy \(\lambda_- \le \lambda_k \le \lambda_+\) for
constants \(\lambda_-, \lambda_+ > 0\). By rotational equivariance,
\(\pmb\Gamma_{\mathrm{SSCM}}\) is a function of \(\pmb\Sigma\) and
therefore shares its eigenvectors; let \(\xi_k\) denote its eigenvalue
on the eigenvector of \(\lambda_k\), so that
\(\|\pmb\Gamma_{\mathrm{SSCM}} - \pmb\Sigma\|_F^2 = \sum_k (\xi_k - \lambda_k)^2\).

\[
\begin{aligned}
\|\pmb \Gamma_{\text{SSCM}} - \pmb\Sigma\|_F^2 &= \sum_{i=1}^p (\xi_i - \lambda_i)^2 \\ 
&= \sum_{i=1}^p \bigg(\lambda_i - \frac{2}{p}\lambda_i^2 + \frac{2}{p^2}\lambda_i\bigg[\sum_{j=1}^p\lambda_j^2\bigg] +  o(1/p) - \lambda_i\bigg)^2 \\
&= \sum_{i=1}^p \bigg| - \frac{2}{p}\lambda_i^2 + \frac{2}{p^2}\lambda_i\bigg[\sum_{j=1}^p\lambda_j^2\bigg] + o(1/p) \bigg|^2\\
&\leq  \sum_{i=1}^p \bigg(\frac{2}{p}\lambda_+^2 + \frac{2\lambda_+^3}{p} + o(1/p) \bigg)^2\\
&= \sum_{i=1}^p \mathcal O (1/p^2)\\
&= \mathcal O(1/p).
\end{aligned}
\]

Above, the expansion of the population SSCM eigenvalues \(\xi_i\) comes
from {[}6{]}, who show
\(\|\pmb \Gamma_{\text{SSCM}}- \mathbf T\|_{\text{op}} = o(1/p)\). The
eigenvalues of \(\mathbf T\) are provided in Equation S2.1 in the
supplement of {[}13{]}.
\end{proof}

\begin{lemma}[Haar Integration of Trace
Products]\protect\hypertarget{lem-haar}{}\label{lem-haar}

Let \(\mathbf A, \mathbf B \in \mathbb R^{p\times p}\) be deterministic
symmetric matrices with trace zero. Let
\(\pmb \Omega = \mathbf U \mathbf D \mathbf U^\top \in \mathbb R^{p\times p}\)
be a random matrix with \(\mathbf U\) distributed as Haar on the
orthogonal group in dimension \(p\), with \(\mathbf D\) diagonal,
independent of \(\mathbf U\), and \(\text{tr}(\pmb \Omega)=0\) almost
surely. Then
\(\mathbb E[\text{tr}(\mathbf A\pmb \Omega)\text{tr}(\mathbf B \pmb \Omega)] = \frac{2 \text{tr}(\mathbf A \mathbf B) \mathbb E\|\pmb \Omega \|_F^2}{(p-1)(p+2)}\).

\begin{proof}
This follows from expanding this product of traces and applying Lemma
2.22 in {[}19{]} to exploit symmetries of the Haar measure and
convenient cancellation from the traceless matrices after simple albeit
tedious manipulations. Below, \(\delta_{ij}\) denotes the Kronecker
delta.

First, write \[
\begin{aligned}
\mathbb E[\text{tr}(\mathbf A\pmb \Omega)\text{tr}(\mathbf B \pmb \Omega)] &= \mathbb E[\text{tr}(\mathbf A \mathbf U \mathbf D \mathbf U^\top)\text{tr}(\mathbf B \mathbf U \mathbf D \mathbf U^\top)] \\
&=\mathbb E\bigg[\sum_{r,i} A_{ri}(\mathbf U \mathbf D \mathbf U^\top)_{ir}\sum_{\lambda, \alpha} B_{\lambda \alpha}(\mathbf U \mathbf D \mathbf U^\top)_{\alpha \lambda}\bigg]\\
&= \mathbb E\bigg[\sum_{r,i}\sum_j  A_{ri}  U_{ij}  D_{jj}  U_{rj} \sum_{\lambda, \alpha} \sum_\mu  B_{\lambda \alpha}  U_{\alpha \mu}  D_{\mu\mu}  U_{\lambda\mu} \bigg]\\
&= \mathbb E\bigg[\sum_{r,i, \lambda, \alpha}\sum_{j,\mu}  A_{ri}  U_{ij}  D_{jj}  U_{rj}  B_{\lambda \alpha}  U_{\alpha \mu}  D_{\mu\mu}  U_{\lambda\mu} \bigg].
\end{aligned}
\]

Then, since \(\mathbf D\) is independent of \(\mathbf U\),

\begin{equation}\phantomsection\label{eq-haar-intermediate}{
\begin{aligned}
\mathbb E[\text{tr}( \mathbf A\pmb \Omega)\text{tr}( \mathbf B \pmb \Omega)] &= \sum_{r,i, \lambda,\alpha}\sum_{j,\mu}  A_{ri}   B_{\lambda \alpha}  \mathbb E [ D_{jj}  D_{\mu\mu}]\mathbb E[  U_{ij}   U_{rj}  U_{\alpha \mu}  U_{\lambda\mu}].
\end{aligned}
}\end{equation}

Applying Meckes' Lemma 2.22 with indices \(j=s\) and \(\beta = \mu\) in
the original notation,

\[
\begin{aligned}
\mathbb E[  U_{ij}   U_{rj}  U_{\alpha \mu}  U_{\lambda\mu}] &=\frac{p+1}{(p-1)p(p+2)}[\delta_{ir}\delta_{\alpha\lambda}\delta_{jj}\delta_{\mu\mu}+\delta_{i\alpha}\delta_{r \lambda}\delta_{j\mu}\delta_{j\mu}+\delta_{i\lambda}\delta_{r\alpha}\delta_{j\mu}\delta_{j\mu}] \\
& \qquad
- \frac{1}{(p-1)p(p+2)}[\delta_{ir}\delta_{\alpha\lambda}\delta_{j\mu}\delta_{j\mu}+\delta_{ir}\delta_{\alpha\lambda}\delta_{j\mu}\delta_{j\mu}+\delta_{i\alpha}\delta_{r \lambda}\delta_{jj}\delta_{\mu\mu} \\
&\qquad +\delta_{i\alpha}\delta_{r\lambda}\delta_{j\mu}\delta_{j\mu}+\delta_{i\lambda}\delta_{r\alpha}\delta_{jj}\delta_{\mu\mu}+\delta_{i\lambda}\delta_{r\alpha}\delta_{j\mu}\delta_{j\mu}]\\
&=  \frac{p+1}{(p-1)p(p+2)}[\delta_{ir}\delta_{\alpha\lambda}+\delta_{i\alpha}\delta_{r \lambda}\delta_{j\mu}+\delta_{i\lambda}\delta_{r\alpha}\delta_{j\mu}] 
\\
& \qquad - \frac{1}{(p-1)p(p+2)}[\delta_{ir}\delta_{\alpha\lambda}\delta_{j\mu}+\delta_{ir}\delta_{\alpha\lambda}\delta_{j\mu}+\delta_{i\alpha}\delta_{r \lambda}+\delta_{i\alpha}\delta_{r\lambda}\delta_{j\mu}+\delta_{i\lambda}\delta_{r\alpha}+\delta_{i\lambda}\delta_{r\alpha}\delta_{j\mu}]\\
 &= \frac{1}{(p-1)p(p+2)}[(p+1)\delta_{ir}\delta_{\alpha\lambda}\\ &\qquad + p(\delta_{i\alpha}\delta_{r \lambda}\delta_{j\mu}+\delta_{i\lambda}\delta_{r\alpha}\delta_{j\mu}) 
- 2(\delta_{ir}\delta_{\alpha\lambda}\delta_{j\mu})-\delta_{i\alpha}\delta_{r \lambda}-\delta_{i\lambda}\delta_{r\alpha}].
\end{aligned}
\]

Plugging this back into (\ref{eq-haar-intermediate}), we can take
advantage of the fact that any terms that include factors
\(\text{tr}(\mathbf A), \text{tr}(\mathbf B)\), or
\(\mathbb E(\text{tr}(\mathbf D)^2)\) are zero. We start by constraining
the summation to be over indices allowed to remain nonzero by the
Kronecker deltas:

\[
\begin{aligned}
\mathbb E[\text{tr}( \mathbf A\pmb \Omega)\text{tr}( \mathbf B \pmb \Omega)]  &=\frac{1}{(p-1)p(p+2)}\bigg( (p+1)\sum_{r,i, \lambda,\alpha}\sum_{j,\mu}  A_{ri}   B_{\lambda \alpha} \mathbb E [ D_{jj}  D_{\mu\mu}]\delta_{ir}\delta_{\alpha\lambda} \\
& \qquad
 + p\sum_{r,i, \lambda,\alpha}\sum_{j,\mu}  A_{ri}   B_{\lambda \alpha} \mathbb E [D_{jj}  D_{\mu\mu}](\delta_{i\alpha}\delta_{r \lambda}\delta_{j\mu}+\delta_{i\lambda}\delta_{r\alpha}\delta_{j\mu}) \\
 &\qquad
- 2\sum_{r,i, \lambda,\alpha}\sum_{j,\mu}  A_{ri}   B_{\lambda \alpha}   \mathbb E [D_{jj}  D_{\mu\mu}](\delta_{ir}\delta_{\alpha\lambda}\delta_{j\mu})  - \sum_{r,i, \lambda,\alpha}\sum_{j,\mu}  A_{ri}   B_{\lambda \alpha}   \mathbb E [\ D_{jj}  D_{\mu\mu}]\delta_{i\alpha}\delta_{r \lambda} \\
&\qquad - \sum_{r,i, \lambda,\alpha}\sum_{j,\mu}  A_{ri}   B_{\lambda \alpha}  \mathbb E [D_{jj} D_{\mu\mu}]\delta_{i\lambda}\delta_{r\alpha}\bigg) \\
&=\frac{1}{(p-1)p(p+2)}\bigg( (p+1)\sum_{r=i, \lambda=\alpha} A_{ri} B_{\lambda \alpha} \sum_{j,\mu}  \mathbb E [D_{jj} D_{\mu\mu}] \\
& \qquad + p\sum_{r,i, \lambda,\alpha} A_{ri}  B_{\lambda \alpha} (\delta_{i\alpha}\delta_{r \lambda}+\delta_{i\lambda}\delta_{r\alpha})  \sum_{j=\mu}  \mathbb E [D_{jj} D_{\mu\mu}] - 2\sum_{r=i, \lambda=\alpha} A_{ri} B_{\lambda \alpha} \sum_{j=\mu}  \mathbb E [D_{jj} D_{\mu\mu}] \\
& \qquad -\sum_{i=\alpha,r=\lambda,} A_{ri}  B_{\lambda \alpha} \sum_{j,\mu}  \mathbb E [D_{jj} D_{\mu\mu}]-\sum_{i= \lambda,r=\alpha} A_{ri} B_{\lambda \alpha} \sum_{j,\mu}  \mathbb E [D_{jj} D_{\mu\mu}]\bigg)\\
&=\frac{1}{(p-1)p(p+2)}\bigg( (p+1)\sum_{r=i, \lambda=\alpha} A_{ri} B_{\lambda \alpha} \mathbb E(\text{tr}(\mathbf D)^2) \\
& \qquad + p\sum_{r,i, \lambda,\alpha} A_{ri} B_{\lambda \alpha} (\delta_{i\alpha}\delta_{r \lambda}+\delta_{i\lambda}\delta_{r\alpha})  \sum_{j=\mu}  \mathbb E [D_{jj} D_{\mu\mu}]
- 2\text{tr}( \mathbf A)  \text{tr}(\mathbf B) \sum_{j=\mu}  \mathbb E [D_{jj} D_{\mu\mu}] \\
& \qquad -\sum_{i=\alpha,r=\lambda,}A_{ri} B_{\lambda \alpha} \mathbb E(\text{tr}(\mathbf D)^2) -\sum_{i= \lambda,r=\alpha}A_{ri} B_{\lambda \alpha} \mathbb E(\text{tr}(\mathbf D)^2)\bigg)\\
&=\frac{1}{(p-1)p(p+2)}\bigg( 
p\sum_{r,i, \lambda,\alpha} A_{ri} B_{\lambda \alpha} (\delta_{i\alpha}\delta_{r \lambda}+\delta_{i\lambda}\delta_{r\alpha})  \sum_{j}  \mathbb E [D_{jj}^2]\bigg)\\
&=\frac{2}{(p-1)(p+2)}\bigg( 
\sum_{r,i} A_{ri} B_{ri}   \sum_{j}  \mathbb E [D_{jj}^2]
\bigg)\\
& = \frac{2 \text{tr}(\mathbf A \mathbf B) \mathbb E\|\pmb \Omega \|_F^2}{(p-1)(p+2)}.
\end{aligned}
\]
\end{proof}

\end{lemma}

\begin{lemma}[Second-Moment Control for
TME]\protect\hypertarget{lem-tme-moments}{}\label{lem-tme-moments}

Let \(\hat{\pmb\Sigma}_{\text{TME}}\) and
\(\hat{\pmb\Sigma}_{\text{SSCM}}\) denote Tyler's M-estimator and the
SSCM, respectively, for data generated according to (\ref{eq-model}),
with Assumption~\ref{exr-hd-regime} of this paper satisfied and
\(c_p, c \in (0,1)\). Then
\(\text{tr}(\mathbb E(\hat{\pmb\Sigma}_{\text{TME}}^2)) = \mathcal O(p)\).

\begin{proof}
Observe that
\(\text{tr}(\mathbb E(\hat{\pmb\Sigma}_{\text{TME}}^2)) = \mathbb E\Pi_n + \text{tr}(\mathbb E(\hat{\pmb\Sigma}_{\text{SSCM}}^2))\).
By Proposition~\ref{prp-tme-l2}, we have
\(\mathbb E \Pi_n =\mathcal O(1)\). By the exact identity given in the
proof of Lemma~\ref{lem-errors}(a) and the characterization of the
eigenvalues of \(\pmb\Gamma_{\mathrm{SSCM}}\) provided by
Lemma~\ref{lem-pop-sscm-risk}, we have

\[
\begin{aligned}
\mathbb E\operatorname{tr}(\hat{\pmb\Sigma}_{\mathrm{SSCM}}^2)&=pc_p+\left(1-\frac{1}{n}\right)\operatorname{tr}(\pmb\Gamma_{\mathrm{SSCM}}^2)\\
&= pc_p+\left(1-\frac{1}{n}\right)\bigg(\operatorname{tr}(\pmb\Sigma^2)+\mathcal O(1)\bigg)\\
&\leq pc_p+\left(1-\frac{1}{n}\right)\bigg(p\lambda_+^2+\mathcal O(1)\bigg)\\
&=\mathcal O(p).
\end{aligned}
\] Thus,
\(\text{tr}(\mathbb E(\hat{\pmb\Sigma}_{\text{TME}}^2)) = \mathcal O(1) + \mathcal O(p)=\mathcal O(p)\).
\end{proof}

\end{lemma}

\begin{lemma}[Frobenius Bias and Cross-Term Control for
TME]\protect\hypertarget{lem-tme-control}{}\label{lem-tme-control}

Assume Tyler's M-estimator exists, and Assumption~\ref{exr-hd-regime}
holds with \(c_p, c <1\). Then

\begin{enumerate}
\def\labelenumi{\alph{enumi})}
\item
  \(\|\mathbb E (\hat{\pmb\Sigma}_{\text{TME}}) - \pmb\Sigma\|_F^2 = \mathcal O(1/p)\)
  and
\item
  \(\text{tr}(\mathbb E (\hat{\pmb\Sigma}_{\text{TME}})\pmb\Sigma) = \text{tr}(\pmb\Sigma^2) + \mathcal O(1/p)\).
\end{enumerate}

\begin{proof}
Proving both statements relies on recasting quantities as traceless
matrices and leveraging Lemma~\ref{lem-haar} for control of the trace
products involved.

First, it is convenient to leverage the affine equivariance and
distribution-free properties of Tyler's M-estimator for shape to write
\(\hat{\pmb\Sigma}_{\text{TME}}(\mathbf X)\) as a function of the
estimator applied to the whitened data matrix
\(\mathbf Z := \mathbf G\pmb\Sigma^{-1/2}\). Write

\[
\hat{\pmb\Sigma}_{\text{TME}}(\mathbf X) = \frac{\pmb\Sigma^{1/2}\hat{\pmb\Sigma}_0\pmb\Sigma^{1/2}}{\frac{1}{p}\text{tr}(\pmb\Sigma \hat{\pmb\Sigma}_0)},
\]

where
\(\hat{\pmb\Sigma}_0 := \hat{\pmb\Sigma}_{\text{TME}}(\mathbf Z)\).

As a simple first case, consider when the true shape is spherical:
\(\pmb\Sigma = \mathbf I\).

Because Tyler's M-estimator is affine equivariant for shape and the
distribution of \(\mathbf Z\) is rotationally invariant, we must have
\(\mathbb E(\hat{\pmb\Sigma}_0) = \mathbf O \mathbb E(\hat{\pmb\Sigma}_0) \mathbf O^\top\)
for any orthogonal matrix \(\mathbf O \in \mathbb R^{p\times p}\). By
Schur's Lemma, the only matrix that is invariant under congruence
transforms with any orthogonal matrix is a multiple of the identity, so
\(\mathbb E(\hat{\pmb\Sigma}_0) \propto \mathbf I\). Since
\(\text{tr}(\hat{\pmb\Sigma}_0)=p\) a.s., this means TME is exactly
unbiased when \(\pmb\Sigma = \mathbf I\).

Next, we extend to the case of a non-spherical shape matrix.

Define the matrices \(\pmb \Delta := \pmb\Sigma - \mathbf I\),
\(\hat{\mathbf E} :=\hat{\pmb\Sigma}_0-\mathbf I\),
\(\mathbf B := \pmb\Sigma^2 - p^{-1}\text{tr}(\pmb\Sigma^2)\mathbf I\),
and the scalar \(u := p^{-1}\text{tr}(\pmb\Delta\hat{\mathbf E} )\). By
the uniform bounds on the eigenvalues of \(\pmb\Sigma\), the squared
Frobenius norms of the deterministic traceless matrices \(\pmb\Delta\)
and \(\mathbf B\) are \(\mathcal O(p)\). Additionally, by
Lemma~\ref{lem-tme-moments} and the fact that
\(\operatorname{tr}(\hat{\pmb\Sigma}_0)=p\), we have
\(\mathbb E \|\hat{\mathbf E}\|_F^2 =\mathcal O(p)\).

By construction,
\(\text{tr}(\pmb \Delta) =\text{tr}(\hat{\mathbf E})=\text{tr}(\mathbf B) =0\)
almost surely, and \(\mathbb E(\hat{\mathbf E}) =\mathbf 0\). Likewise,
\(\mathbb E(u)=0\). Recall that from the Assumption~\ref{exr-hd-regime},
\(\pmb\Sigma\) has uniformly bounded eigenvalues
\(\lambda_- \leq \lambda_k(\pmb\Sigma) \leq \lambda_+\) for
\(k = 1, \dots, p\). Then we have \(\lambda_-\) as a lower bound on
\(1+u=p^{-1}\text{tr}(\pmb\Sigma \hat{\pmb\Sigma}_0)\):\\
\begin{equation}\phantomsection\label{eq-one-plus-u}{
\begin{aligned}
\lambda_-=\frac{\lambda_-}{p}\text{tr}(\hat{\pmb\Sigma}_0) &=\frac{1}{p}\text{tr}(\lambda_-\mathbf I \hat{\pmb\Sigma}_0) \\
&\leq \frac{1}{p}\text{tr}(\pmb\Sigma \hat{\pmb\Sigma}_0)\\
&= \frac{1}{p}\text{tr}((\pmb \Delta + \mathbf I)(\hat{\mathbf E}+\mathbf I)) \\
&= 1+u.
\end{aligned}
}\end{equation}

Finally, because \(\hat{\pmb\Sigma}_0\) has a rotationally invariant
distribution, its eigenvectors are Haar-distributed on the orthogonal
group in dimension \(p\) and independent of its eigenvalues {[}20{]}.
Thus, because \(\hat{\mathbf E} =\hat{\pmb\Sigma}_0-\mathbf I\), it
shares these same Haar-distributed eigenvectors, while its diagonal
matrix of eigenvalues is simply shifted by the identity. Consequently,
\(\mathbb E(u^2) = \frac{1}{p^2} \mathbb E(\text{tr}(\pmb\Delta\hat{\mathbf E} )\text{tr}(\pmb\Delta\hat{\mathbf E} )) = \frac{1}{p^2} \frac{2 \|\pmb \Delta\|_F^2 \mathbb E\|\hat{\mathbf E}\|_F^2}{(p-1)(p+2)} = \mathcal O(1/p^2)\)
immediately by Lemma~\ref{lem-haar}.

\begin{enumerate}
\def\labelenumi{\alph{enumi})}
\tightlist
\item
  Starting with
  \(\|\mathbb E (\hat{\pmb\Sigma}_{\text{TME}}) - \pmb\Sigma\|_F^2 = \mathcal O(1/p)\).
\end{enumerate}

First we use the unbiasedness of \(\hat{\pmb\Sigma}_0\) to write both
terms as expectations since
\(\mathbb E[\pmb\Sigma^{1/2}\hat{\pmb\Sigma}_0\pmb\Sigma^{1/2}] = \pmb\Sigma\):

\[
\begin{aligned}
\|\mathbb E(\hat{\pmb\Sigma}_{\text{TME}}) - \pmb\Sigma\|_F^2 &= \bigg\| \mathbb E\bigg[\frac{\pmb\Sigma^{1/2}\hat{\pmb\Sigma}_0\pmb\Sigma^{1/2}}{\frac{1}{p}\text{tr}(\pmb\Sigma \hat{\pmb\Sigma}_0)} \bigg] - \mathbb E[\pmb\Sigma^{1/2}\hat{\pmb\Sigma}_0\pmb\Sigma^{1/2}] \bigg\|_F^2 \\
&= \bigg\| \mathbb E\bigg[\frac{\pmb\Sigma^{1/2}\hat{\pmb\Sigma}_0\pmb\Sigma^{1/2}}{\frac{1}{p}\text{tr}(\pmb\Sigma \hat{\pmb\Sigma}_0)} - \pmb\Sigma^{1/2}\hat{\pmb\Sigma}_0\pmb\Sigma^{1/2} \bigg]\bigg\|_F^2 \\
&= \bigg\| \mathbb E\bigg[ \pmb\Sigma^{1/2}\hat{\pmb\Sigma}_0\pmb\Sigma^{1/2} \bigg(\frac{1}{\frac{1}{p}\text{tr}(\pmb\Sigma \hat{\pmb\Sigma}_0)}-1\bigg) \bigg]\bigg\|_F^2.
\end{aligned}
\]

Next, using (\ref{eq-one-plus-u}) to write the trace in the denominator
in terms of \(u\):

\[
\begin{aligned}
\|\mathbb E(\hat{\pmb\Sigma}_{\text{TME}}) - \pmb\Sigma\|_F^2 &= \bigg\| \mathbb E\bigg[ \pmb\Sigma^{1/2}\hat{\pmb\Sigma}_0\pmb\Sigma^{1/2} \bigg(\frac{1}{1+u}-1\bigg) \bigg]\bigg\|_F^2 \\
&= \bigg(\bigg\| \mathbb E\bigg[ \pmb\Sigma^{1/2}\hat{\pmb\Sigma}_0\pmb\Sigma^{1/2} \bigg(\frac{-u}{1+u}\bigg) \bigg]\bigg\|_F\bigg)^2 \\
&\leq \bigg( \mathbb E\bigg\| \pmb\Sigma^{1/2}\hat{\pmb\Sigma}_0\pmb\Sigma^{1/2} \bigg(\frac{-u}{1+u}\bigg) \bigg\|_F\bigg)^2 & (*)\\
&= \bigg( \mathbb E\bigg\{\bigg\| \pmb\Sigma^{1/2}\hat{\pmb\Sigma}_0\pmb\Sigma^{1/2}\bigg\|_F\bigg|\frac{u}{1+u}\bigg|\bigg\}\bigg)^2 \\
&\leq \mathbb E \|\pmb\Sigma^{1/2}\hat{\pmb\Sigma}_0\pmb\Sigma^{1/2} \|_F^2\ \mathbb E \bigg[\bigg( \frac{u}{1+u} \bigg)^2\bigg] & (**)\\
&\leq \mathbb E \|\pmb\Sigma^{1/2}\hat{\pmb\Sigma}_0\pmb\Sigma^{1/2} \|_F^2\ \frac{1}{\lambda_-^2}\mathbb E [u^2] \\
&\leq \frac{1}{\lambda_-^2}\|\pmb\Sigma\|_{\text{op}}^2 \mathbb E\| \hat{\pmb\Sigma}_0 \|_F^2\mathbb E [u^2] \\
&\leq \frac{\lambda_+^2}{\lambda_-^2} \text{tr}(\mathbb E(\hat{\pmb\Sigma}_0^2))\mathbb E [u^2] \\
&= \mathcal O(p)\mathcal O(1/p^2)\\
&= \mathcal O (1/p).
\end{aligned}
\]

Where \((*)\) follows from Jensen's Inequality and \((**)\) from the
Cauchy-Schwarz inequality. The inequality that follows on the subsequent
line is a direct consequence of (\ref{eq-one-plus-u}) and its bound
\(1+u \geq \lambda_-\), and
\(\text{tr}(\mathbb E(\hat{\pmb\Sigma}_0^2)) = \mathcal O(p)\) is shown
in Lemma~\ref{lem-tme-moments}.

\begin{enumerate}
\def\labelenumi{\alph{enumi})}
\setcounter{enumi}{1}
\tightlist
\item
  Next,
  \(\text{tr}(\mathbb E (\hat{\pmb\Sigma}_{\text{TME}})\pmb\Sigma) = \text{tr}(\pmb\Sigma^2) + \mathcal O(1/p)\)
\end{enumerate}

\[
\begin{aligned}
\text{tr}(\hat{\pmb\Sigma}_{\text{TME}}\pmb\Sigma) &= \frac{\text{tr}(\pmb\Sigma^{1/2}\hat{\pmb\Sigma}_0\pmb\Sigma^{1/2}\pmb\Sigma)}{\frac{1}{p}\text{tr}(\pmb\Sigma \hat{\pmb\Sigma}_0)} \\
&= \frac{\text{tr}(\pmb\Sigma^{2}\hat{\pmb\Sigma}_0)}{\frac{1}{p}\text{tr}(\pmb\Sigma \hat{\pmb\Sigma}_0)} \\
&= \frac{\text{tr}((\mathbf B + \frac{1}{p}\text{tr}(\pmb\Sigma^2)\mathbf I)(\hat{\mathbf E} + \mathbf I))}{\frac{1}{p}\text{tr}(\pmb\Sigma \hat{\pmb\Sigma}_0)} \\
&= \frac{1}{\frac{1}{p}\text{tr}(\pmb\Sigma \hat{\pmb\Sigma}_0)}\text{tr}[(\mathbf B\hat{\mathbf E})+\mathbf B + \frac{1}{p}\text{tr}(\pmb\Sigma^2)\hat{\mathbf E} + \frac{1}{p}\text{tr}(\pmb\Sigma^2)\mathbf I]\\
&= \frac{\text{tr}(\mathbf B\hat{\mathbf E})}{\frac{1}{p}\text{tr}(\pmb\Sigma \hat{\pmb\Sigma}_0)} + \frac{\text{tr}(\pmb\Sigma^2)}{\frac{1}{p}\text{tr}(\pmb\Sigma \hat{\pmb\Sigma}_0)} \\
&= \frac{1}{1+u}(\text{tr}(\mathbf B\hat{\mathbf E})+\text{tr}(\pmb\Sigma^2))\\
&= \bigg(1-u+\frac{u^2}{1+u}\bigg)(\text{tr}(\mathbf B\hat{\mathbf E})+\text{tr}(\pmb\Sigma^2)) \\
&= \text{tr}(\pmb\Sigma^2)+\text{tr}(\mathbf B \hat{\mathbf E}) -u\text{tr}(\pmb\Sigma^2) - u\text{tr}(\mathbf B \hat{\mathbf E}) +\frac{u^2}{1+u}[\text{tr}(\pmb\Sigma^2) +\text{tr}(\mathbf B \hat{\mathbf E})].
\end{aligned}
\]

Taking absolute deviations of expectations, since
\(\mathbb E\text{tr}(\mathbf B \hat{\mathbf E})=0\) and
\(\mathbb E(u\text{tr}(\pmb\Sigma^2))=0\), we can write

\[
\begin{aligned}
|\text{tr}(\mathbb E (\hat{\pmb\Sigma}_{\text{TME}})\pmb\Sigma) - \text{tr}(\pmb\Sigma^2) |&\leq 0 + 0 + |\mathbb E(u\text{tr}(\mathbf B \hat{\mathbf E}))| + \bigg|\mathbb E \bigg[\frac{u^2}{1+u}[\text{tr}(\pmb\Sigma^2) +\text{tr}(\mathbf B \hat{\mathbf E})]\bigg]\bigg| \\
&= \bigg|\frac{1}{p}\mathbb E[\text{tr}(\pmb \Delta \hat{\mathbf E})\text{tr}(\mathbf B \hat{\mathbf E})]\bigg|  +\bigg|\mathbb E \bigg[\frac{u^2}{1+u}[\text{tr}(\pmb\Sigma^2) +\text{tr}(\mathbf B \hat{\mathbf E})]\bigg]\bigg| \\
& \leq  \bigg|\frac{1}{p}\mathbb E[\text{tr}(\pmb \Delta \hat{\mathbf E})\text{tr}(\mathbf B \hat{\mathbf E})]\bigg|  +\bigg|\frac{1}{\lambda_-}\mathbb E \bigg[u^2[\text{tr}(\pmb\Sigma^2) +\text{tr}(\mathbf B \hat{\mathbf E})]\bigg]\bigg| \\
&= \bigg|\frac{1}{p}\frac{2 \text{tr}(\pmb \Delta \mathbf B) \mathbb E\|\hat{\mathbf E}\|_F^2}{(p-1)(p+2)}\bigg|+\bigg|\frac{1}{\lambda_-}\mathbb E \bigg[u^2[\text{tr}(\pmb\Sigma^2) +\text{tr}((\pmb\Sigma^2 - \frac{1}{p}\text{tr}(\pmb\Sigma^2)\mathbf I)\hat{\mathbf E})]\bigg]\bigg| \\
&= \bigg|\frac{1}{p}\frac{2 \text{tr}(\pmb \Delta \mathbf B) \mathbb E\|\hat{\mathbf E}\|_F^2}{(p-1)(p+2)}\bigg|+\bigg|\frac{1}{\lambda_-}\mathbb E \bigg[u^2[\text{tr}(\pmb\Sigma^2) +\text{tr}(\pmb\Sigma^2\hat{\mathbf E})]\bigg]\bigg| \\
&= \bigg|\frac{1}{p}\frac{2 \text{tr}(\pmb \Delta \mathbf B) \mathbb E\|\hat{\mathbf E}\|_F^2}{(p-1)(p+2)}\bigg|+\bigg|\frac{1}{\lambda_-}\mathbb E \bigg[u^2[\text{tr}(\pmb\Sigma^2 + \pmb\Sigma^2(\hat{\pmb\Sigma}_0 - \mathbf I))]\bigg]\bigg| \\
&= \bigg|\frac{1}{p}\frac{2 \text{tr}(\pmb \Delta \mathbf B) \mathbb E\|\hat{\mathbf E}\|_F^2}{(p-1)(p+2)}\bigg| +\bigg|\frac{1}{\lambda_-}\mathbb E [u^2\text{tr}(\pmb\Sigma^2\hat{\pmb\Sigma}_0)]\bigg| \\
&\leq \bigg|\frac{1}{p}\frac{2 \text{tr}(\pmb \Delta \mathbf B) \mathbb E\|\hat{\mathbf E}\|_F^2}{(p-1)(p+2)}\bigg|+\frac{1}{\lambda_-}\mathbb E [u^2 \lambda_+^2\text{tr}(\hat{\pmb\Sigma}_0)]  \\
&= \bigg|\frac{1}{p}\frac{2 \text{tr}(\pmb \Delta \mathbf B) \mathbb E\|\hat{\mathbf E}\|_F^2}{(p-1)(p+2)}\bigg|+\frac{\lambda_+^2 p}{\lambda_-}\mathbb E  [u^2] \\
&= \frac{1}{p}\frac{2 \mathcal O(p) \mathcal O(p)}{(p-1)(p+2)}+\frac{\lambda_+^2p}{\lambda_-}\mathcal O(1/p^2)\\
&= \mathcal O(1/p),
\end{aligned}
\]

where the penultimate line uses the fact that
\(|\text{tr}(\pmb\Delta \mathbf B)| \leq \|\pmb \Delta \|_F \| \mathbf B \|_F = \mathcal O(\sqrt p)\mathcal O(\sqrt p) = \mathcal O(p)\).

Thus we have
\(\text{tr}(\mathbb E (\hat{\pmb\Sigma}_{\text{TME}})\pmb\Sigma) = \text{tr}(\pmb\Sigma^2) + \mathcal O(1/p)\).
\end{proof}

\end{lemma}

\newpage

\section{SIMULATION TABLES}\label{simulation-tables}

For ease of comparison across varying \(p\), in the tables below,
theoretical risk is reported as
\(p^{-1} \mathbb E \|\hat{\pmb\Sigma} - \pmb\Sigma\|_F^2\), and
simulated risk is reported as the average similarly scaled risk,
\((Rp)^{-1}\sum_{r=1}^R\|\hat{\boldsymbol{\Sigma}}^{(r)}-\boldsymbol{\Sigma}\|_F^2\),
over \(R=150\) simulations, with \(\hat{\boldsymbol{\Sigma}}^{(r)}\) the
\(r\)th replicate. Monte Carlo standard errors, in parentheses, are
standard errors of the mean over replicates; because the theoretical
risk is a constant, the same standard error applies to Sim. \(-\)
Theory. The SSCM's risk improvement relative to a benchmark is computed
from paired replicates, with a delta-method standard error. Theoretical
values are the approximations of the high-dimensional limits without
finite-\(p\) bias correction. Aspect ratios \(c_p=p/n\) are displayed to
one decimal place; theoretical risks are calculated using the unrounded
ratios.

\begingroup\fontsize{8}{10}\selectfont

\begin{longtable}[t]{lrrrrrr}
\caption{Risk and Runtime by Shape Estimator, $p = 25$}\\
\toprule
Estimator & \makecell[c]{Sim. Risk\\(MC SE)} & \makecell[c]{Theory\\Risk} & \makecell[c]{Sim. $-$\\Theory} & \makecell[c]{Runtime\\(ms)} & \makecell[c]{SSCM Risk\\Improv. \% (SE)} & \makecell[c]{Runtime\\vs SSCM}\\
\midrule
\endfirsthead
\caption[]{Risk and Runtime by Shape Estimator, $p = 25$ \textit{(continued)}}\\
\toprule
Estimator & \makecell[c]{Sim. Risk\\(MC SE)} & \makecell[c]{Theory\\Risk} & \makecell[c]{Sim. $-$\\Theory} & \makecell[c]{Runtime\\(ms)} & \makecell[c]{SSCM Risk\\Improv. \% (SE)} & \makecell[c]{Runtime\\vs SSCM}\\
\midrule
\endhead

\endfoot
\bottomrule
\endlastfoot
\addlinespace[0.3em]
\multicolumn{7}{l}{\textbf{$c_p=0.1, \rho=8$}}\\
\hspace{1em}SSCM & 0.1062 (0.0013) & 0.0906 & 0.0157 & 0.08 & - & -\\
\hspace{1em}Latent SCM & 0.1082 (0.0018) & 0.1094 & -0.0013 & 0.03 & 1.8 (1.1) & -58.5\%\\
\hspace{1em}Tyler's M & 0.1155 (0.0019) & 0.1183 & -0.0029 & 1.39 & 8.01 (0.81) & 1712.6\%\\
\addlinespace[0.3em]
\multicolumn{7}{l}{\textbf{$c_p=0.1, \rho=100$}}\\
\hspace{1em}SSCM & 0.0838 (0.0025) & 0.0815 & 0.0023 & 0.10 & - & -\\
\hspace{1em}Latent SCM & 0.1183 (0.0038) & 0.1185 & -0.0002 & 0.04 & 29.1 (1.6) & -58.1\%\\
\hspace{1em}Tyler's M & 0.1200 (0.0039) & 0.1274 & -0.0074 & 1.93 & 30.2 (1.2) & 1928.6\%\\
\addlinespace[0.3em]
\multicolumn{7}{l}{\textbf{$c_p=0.1, \rho=10000$}}\\
\hspace{1em}SSCM & 0.0782 (0.0029) & 0.0800 & -0.0018 & 0.07 & - & -\\
\hspace{1em}Latent SCM & 0.1194 (0.0043) & 0.1200 & -0.0006 & 0.03 & 34.5 (1.8) & -57.4\%\\
\hspace{1em}Tyler's M & 0.1206 (0.0042) & 0.1289 & -0.0082 & 1.38 & 35.2 (1.2) & 1825.1\%\\
\addlinespace[0.3em]
\multicolumn{7}{l}{\textbf{$c_p=0.5, \rho=8$}}\\
\hspace{1em}SSCM & 0.4753 (0.0059) & 0.4528 & 0.0226 & 0.04 & - & -\\
\hspace{1em}Latent SCM & 0.5534 (0.0097) & 0.5472 & 0.0062 & 0.01 & 14.1 (1.4) & -64.5\%\\
\hspace{1em}Tyler's M & 0.623 (0.011) & 0.627 & -0.005 & 2.12 & 23.65 (0.79) & 4987.3\%\\
\addlinespace[0.3em]
\multicolumn{7}{l}{\textbf{$c_p=0.5, \rho=100$}}\\
\hspace{1em}SSCM & 0.435 (0.011) & 0.408 & 0.027 & 0.04 & - & -\\
\hspace{1em}Latent SCM & 0.608 (0.018) & 0.592 & 0.016 & 0.01 & 28.6 (1.5) & -63.4\%\\
\hspace{1em}Tyler's M & 0.664 (0.019) & 0.672 & -0.008 & 2.12 & 34.6 (1.4) & 5163.7\%\\
\addlinespace[0.3em]
\multicolumn{7}{l}{\textbf{$c_p=0.5, \rho=10000$}}\\
\hspace{1em}SSCM & 0.409 (0.013) & 0.400 & 0.009 & 0.04 & - & -\\
\hspace{1em}Latent SCM & 0.632 (0.022) & 0.600 & 0.032 & 0.01 & 35.2 (2.0) & -63.3\%\\
\hspace{1em}Tyler's M & 0.676 (0.025) & 0.680 & -0.004 & 2.27 & 39.5 (1.7) & 5492.9\%\\
\addlinespace[0.3em]
\multicolumn{7}{l}{\textbf{$c_p=0.9, \rho=8$}}\\
\hspace{1em}SSCM & 0.832 (0.010) & 0.809 & 0.024 & 0.04 & - & -\\
\hspace{1em}Latent SCM & 0.962 (0.018) & 0.977 & -0.015 & 0.01 & 13.5 (1.4) & -61.4\%\\
\hspace{1em}Tyler's M & 1.526 (0.028) & 1.644 & -0.117 & 11.55 & 45.48 (0.98) & 30803.4\%\\
\addlinespace[0.3em]
\multicolumn{7}{l}{\textbf{$c_p=0.9, \rho=100$}}\\
\hspace{1em}SSCM & 0.749 (0.020) & 0.728 & 0.021 & 0.04 & - & -\\
\hspace{1em}Latent SCM & 1.039 (0.038) & 1.058 & -0.019 & 0.01 & 27.9 (2.0) & -65.8\%\\
\hspace{1em}Tyler's M & 1.659 (0.056) & 1.725 & -0.066 & 11.28 & 54.9 (1.5) & 29201.6\%\\
\addlinespace[0.3em]
\multicolumn{7}{l}{\textbf{$c_p=0.9, \rho=10000$}}\\
\hspace{1em}SSCM & 0.678 (0.019) & 0.714 & -0.036 & 0.04 & - & -\\
\hspace{1em}Latent SCM & 1.015 (0.035) & 1.071 & -0.057 & 0.01 & 33.2 (2.2) & -65.1\%\\
\hspace{1em}Tyler's M & 1.601 (0.057) & 1.738 & -0.137 & 11.85 & 57.6 (1.5) & 29727.0\%\\*
\end{longtable}
\endgroup{}

\begingroup\fontsize{8}{10}\selectfont

\begin{longtable}[t]{lrrrrrr}
\caption{Risk and Runtime by Shape Estimator, $p = 50$}\\
\toprule
Estimator & \makecell[c]{Sim. Risk\\(MC SE)} & \makecell[c]{Theory\\Risk} & \makecell[c]{Sim. $-$\\Theory} & \makecell[c]{Runtime\\(ms)} & \makecell[c]{SSCM Risk\\Improv. \% (SE)} & \makecell[c]{Runtime\\vs SSCM}\\
\midrule
\endfirsthead
\caption[]{Risk and Runtime by Shape Estimator, $p = 50$ \textit{(continued)}}\\
\toprule
Estimator & \makecell[c]{Sim. Risk\\(MC SE)} & \makecell[c]{Theory\\Risk} & \makecell[c]{Sim. $-$\\Theory} & \makecell[c]{Runtime\\(ms)} & \makecell[c]{SSCM Risk\\Improv. \% (SE)} & \makecell[c]{Runtime\\vs SSCM}\\
\midrule
\endhead

\endfoot
\bottomrule
\endlastfoot
\addlinespace[0.3em]
\multicolumn{7}{l}{\textbf{$c_p=0.1, \rho=8$}}\\
\hspace{1em}SSCM & 0.10114 (0.00073) & 0.09528 & 0.00586 & 0.24 & - & -\\
\hspace{1em}Latent SCM & 0.10691 (0.00089) & 0.10472 & 0.00219 & 0.12 & 5.40 (0.43) & -47.2\%\\
\hspace{1em}Tyler's M & 0.11083 (0.00093) & 0.10917 & 0.00166 & 4.09 & 8.74 (0.28) & 1628.8\%\\
\addlinespace[0.3em]
\multicolumn{7}{l}{\textbf{$c_p=0.1, \rho=100$}}\\
\hspace{1em}SSCM & 0.0938 (0.0012) & 0.0908 & 0.0031 & 0.22 & - & -\\
\hspace{1em}Latent SCM & 0.1129 (0.0016) & 0.1092 & 0.0037 & 0.15 & 16.91 (0.72) & -34.3\%\\
\hspace{1em}Tyler's M & 0.1162 (0.0016) & 0.1137 & 0.0025 & 4.42 & 19.24 (0.59) & 1886.4\%\\
\addlinespace[0.3em]
\multicolumn{7}{l}{\textbf{$c_p=0.1, \rho=10000$}}\\
\hspace{1em}SSCM & 0.0893 (0.0014) & 0.0900 & -0.0007 & 0.22 & - & -\\
\hspace{1em}Latent SCM & 0.1075 (0.0017) & 0.1100 & -0.0025 & 0.15 & 16.91 (0.82) & -31.8\%\\
\hspace{1em}Tyler's M & 0.1111 (0.0018) & 0.1144 & -0.0033 & 4.21 & 19.63 (0.62) & 1846.5\%\\
\addlinespace[0.3em]
\multicolumn{7}{l}{\textbf{$c_p=0.5, \rho=8$}}\\
\hspace{1em}SSCM & 0.4849 (0.0032) & 0.4764 & 0.0085 & 0.08 & - & -\\
\hspace{1em}Latent SCM & 0.5327 (0.0050) & 0.5236 & 0.0091 & 0.04 & 8.98 (0.66) & -49.7\%\\
\hspace{1em}Tyler's M & 0.5661 (0.0043) & 0.5636 & 0.0025 & 4.11 & 14.35 (0.32) & 5339.8\%\\
\addlinespace[0.3em]
\multicolumn{7}{l}{\textbf{$c_p=0.5, \rho=100$}}\\
\hspace{1em}SSCM & 0.4500 (0.0061) & 0.4538 & -0.0037 & 0.07 & - & -\\
\hspace{1em}Latent SCM & 0.5357 (0.0087) & 0.5462 & -0.0106 & 0.04 & 15.99 (0.92) & -51.3\%\\
\hspace{1em}Tyler's M & 0.5674 (0.0080) & 0.5862 & -0.0188 & 4.30 & 20.69 (0.57) & 5669.6\%\\
\addlinespace[0.3em]
\multicolumn{7}{l}{\textbf{$c_p=0.5, \rho=10000$}}\\
\hspace{1em}SSCM & 0.4495 (0.0076) & 0.4500 & -0.0006 & 0.07 & - & -\\
\hspace{1em}Latent SCM & 0.5533 (0.0093) & 0.5500 & 0.0033 & 0.04 & 18.77 (0.88) & -51.2\%\\
\hspace{1em}Tyler's M & 0.5891 (0.0099) & 0.5900 & -0.0008 & 4.28 & 23.71 (0.71) & 5798.2\%\\
\addlinespace[0.3em]
\multicolumn{7}{l}{\textbf{$c_p=0.9, \rho=8$}}\\
\hspace{1em}SSCM & 0.8646 (0.0061) & 0.8507 & 0.0139 & 0.06 & - & -\\
\hspace{1em}Latent SCM & 0.9445 (0.0100) & 0.9350 & 0.0094 & 0.03 & 8.45 (0.71) & -54.8\%\\
\hspace{1em}Tyler's M & 1.287 (0.015) & 1.268 & 0.018 & 19.10 & 32.80 (0.63) & 32683.5\%\\
\addlinespace[0.3em]
\multicolumn{7}{l}{\textbf{$c_p=0.9, \rho=100$}}\\
\hspace{1em}SSCM & 0.823 (0.011) & 0.810 & 0.013 & 0.06 & - & -\\
\hspace{1em}Latent SCM & 0.985 (0.017) & 0.975 & 0.009 & 0.03 & 16.5 (1.1) & -52.3\%\\
\hspace{1em}Tyler's M & 1.313 (0.021) & 1.309 & 0.005 & 18.53 & 37.35 (0.83) & 31687.3\%\\
\addlinespace[0.3em]
\multicolumn{7}{l}{\textbf{$c_p=0.9, \rho=10000$}}\\
\hspace{1em}SSCM & 0.782 (0.012) & 0.804 & -0.021 & 0.06 & - & -\\
\hspace{1em}Latent SCM & 0.954 (0.017) & 0.982 & -0.028 & 0.03 & 18.0 (1.1) & -57.5\%\\
\hspace{1em}Tyler's M & 1.277 (0.021) & 1.315 & -0.038 & 18.47 & 38.75 (0.97) & 29532.9\%\\*
\end{longtable}
\endgroup{}

\begingroup\fontsize{8}{10}\selectfont

\begin{longtable}[t]{lrrrrrr}
\caption{Risk and Runtime by Shape Estimator, $p = 100$}\\
\toprule
Estimator & \makecell[c]{Sim. Risk\\(MC SE)} & \makecell[c]{Theory\\Risk} & \makecell[c]{Sim. $-$\\Theory} & \makecell[c]{Runtime\\(ms)} & \makecell[c]{SSCM Risk\\Improv. \% (SE)} & \makecell[c]{Runtime\\vs SSCM}\\
\midrule
\endfirsthead
\caption[]{Risk and Runtime by Shape Estimator, $p = 100$ \textit{(continued)}}\\
\toprule
Estimator & \makecell[c]{Sim. Risk\\(MC SE)} & \makecell[c]{Theory\\Risk} & \makecell[c]{Sim. $-$\\Theory} & \makecell[c]{Runtime\\(ms)} & \makecell[c]{SSCM Risk\\Improv. \% (SE)} & \makecell[c]{Runtime\\vs SSCM}\\
\midrule
\endhead

\endfoot
\bottomrule
\endlastfoot
\addlinespace[0.3em]
\multicolumn{7}{l}{\textbf{$c_p=0.1, \rho=8$}}\\
\hspace{1em}SSCM & 0.09919 (0.00037) & 0.09764 & 0.00155 & 0.96 & - & -\\
\hspace{1em}Latent SCM & 0.10304 (0.00042) & 0.10236 & 0.00068 & 0.66 & 3.73 (0.19) & -31.4\%\\
\hspace{1em}Tyler's M & 0.10500 (0.00041) & 0.10458 & 0.00042 & 18.03 & 5.533 (0.096) & 1769.9\%\\
\addlinespace[0.3em]
\multicolumn{7}{l}{\textbf{$c_p=0.1, \rho=100$}}\\
\hspace{1em}SSCM & 0.09453 (0.00074) & 0.09538 & -0.00085 & 1.04 & - & -\\
\hspace{1em}Latent SCM & 0.10350 (0.00083) & 0.10462 & -0.00112 & 0.61 & 8.67 (0.33) & -41.1\%\\
\hspace{1em}Tyler's M & 0.10559 (0.00083) & 0.10685 & -0.00126 & 18.45 & 10.48 (0.26) & 1668.8\%\\
\addlinespace[0.3em]
\multicolumn{7}{l}{\textbf{$c_p=0.1, \rho=10000$}}\\
\hspace{1em}SSCM & 0.09537 (0.00077) & 0.09500 & 0.00037 & 1.00 & - & -\\
\hspace{1em}Latent SCM & 0.10562 (0.00076) & 0.10500 & 0.00062 & 0.63 & 9.70 (0.40) & -37.5\%\\
\hspace{1em}Tyler's M & 0.10742 (0.00081) & 0.10722 & 0.00020 & 17.93 & 11.21 (0.27) & 1687.4\%\\
\addlinespace[0.3em]
\multicolumn{7}{l}{\textbf{$c_p=0.5, \rho=8$}}\\
\hspace{1em}SSCM & 0.4884 (0.0018) & 0.4882 & 0.0002 & 0.23 & - & -\\
\hspace{1em}Latent SCM & 0.5108 (0.0025) & 0.5118 & -0.0010 & 0.15 & 4.38 (0.29) & -35.9\%\\
\hspace{1em}Tyler's M & 0.5299 (0.0021) & 0.5318 & -0.0019 & 13.79 & 7.83 (0.13) & 5883.4\%\\
\addlinespace[0.3em]
\multicolumn{7}{l}{\textbf{$c_p=0.5, \rho=100$}}\\
\hspace{1em}SSCM & 0.4722 (0.0032) & 0.4769 & -0.0047 & 0.23 & - & -\\
\hspace{1em}Latent SCM & 0.5201 (0.0040) & 0.5231 & -0.0030 & 0.15 & 9.21 (0.43) & -33.5\%\\
\hspace{1em}Tyler's M & 0.5365 (0.0039) & 0.5431 & -0.0067 & 13.76 & 11.98 (0.29) & 5914.3\%\\
\addlinespace[0.3em]
\multicolumn{7}{l}{\textbf{$c_p=0.5, \rho=10000$}}\\
\hspace{1em}SSCM & 0.4773 (0.0042) & 0.4750 & 0.0023 & 0.23 & - & -\\
\hspace{1em}Latent SCM & 0.5296 (0.0051) & 0.5250 & 0.0046 & 0.16 & 9.87 (0.47) & -33.0\%\\
\hspace{1em}Tyler's M & 0.5431 (0.0048) & 0.5450 & -0.0019 & 14.48 & 12.12 (0.32) & 6108.7\%\\
\addlinespace[0.3em]
\multicolumn{7}{l}{\textbf{$c_p=0.9, \rho=8$}}\\
\hspace{1em}SSCM & 0.8837 (0.0032) & 0.8796 & 0.0040 & 0.16 & - & -\\
\hspace{1em}Latent SCM & 0.9268 (0.0054) & 0.9222 & 0.0047 & 0.10 & 4.66 (0.39) & -38.8\%\\
\hspace{1em}Tyler's M & 1.1065 (0.0048) & 1.1040 & 0.0026 & 62.37 & 20.14 (0.24) & 38984.3\%\\
\addlinespace[0.3em]
\multicolumn{7}{l}{\textbf{$c_p=0.9, \rho=100$}}\\
\hspace{1em}SSCM & 0.8564 (0.0065) & 0.8592 & -0.0028 & 0.16 & - & -\\
\hspace{1em}Latent SCM & 0.9394 (0.0082) & 0.9426 & -0.0032 & 0.10 & 8.83 (0.52) & -39.0\%\\
\hspace{1em}Tyler's M & 1.1137 (0.0085) & 1.1244 & -0.0107 & 61.13 & 23.11 (0.43) & 38798.6\%\\
\addlinespace[0.3em]
\multicolumn{7}{l}{\textbf{$c_p=0.9, \rho=10000$}}\\
\hspace{1em}SSCM & 0.8561 (0.0069) & 0.8559 & 0.0002 & 0.16 & - & -\\
\hspace{1em}Latent SCM & 0.9553 (0.0089) & 0.9459 & 0.0094 & 0.10 & 10.38 (0.53) & -38.0\%\\
\hspace{1em}Tyler's M & 1.1264 (0.0090) & 1.1277 & -0.0013 & 62.46 & 24.00 (0.48) & 38698.4\%\\*
\end{longtable}
\endgroup{}

\begingroup\fontsize{8}{10}\selectfont

\begin{longtable}[t]{lrrrrrr}
\caption{Risk and Runtime by Shape Estimator, $p = 1000$}\\
\toprule
Estimator & \makecell[c]{Sim. Risk\\(MC SE)} & \makecell[c]{Theory\\Risk} & \makecell[c]{Sim. $-$\\Theory} & \makecell[c]{Runtime\\(ms)} & \makecell[c]{SSCM Risk\\Improv. \% (SE)} & \makecell[c]{Runtime\\vs SSCM}\\
\midrule
\endfirsthead
\caption[]{Risk and Runtime by Shape Estimator, $p = 1000$ \textit{(continued)}}\\
\toprule
Estimator & \makecell[c]{Sim. Risk\\(MC SE)} & \makecell[c]{Theory\\Risk} & \makecell[c]{Sim. $-$\\Theory} & \makecell[c]{Runtime\\(ms)} & \makecell[c]{SSCM Risk\\Improv. \% (SE)} & \makecell[c]{Runtime\\vs SSCM}\\
\midrule
\endhead

\endfoot
\bottomrule
\endlastfoot
\addlinespace[0.3em]
\multicolumn{7}{l}{\textbf{$c_p=0.1, \rho=8$}}\\
\hspace{1em}SSCM & 0.099721 (0.000040) & 0.099764 & -0.000043 & 584.39 & - & -\\
\hspace{1em}Latent SCM & 0.100202 (0.000043) & 0.100236 & -0.000034 & 496.04 & 0.480 (0.011) & -15.1\%\\
\hspace{1em}Tyler's M & 0.100404 (0.000041) & 0.100458 & -0.000054 & 11769.15 & 0.6807 (0.0032) & 1913.9\%\\
\addlinespace[0.3em]
\multicolumn{7}{l}{\textbf{$c_p=0.1, \rho=100$}}\\
\hspace{1em}SSCM & 0.099584 (0.000078) & 0.099538 & 0.000047 & 568.09 & - & -\\
\hspace{1em}Latent SCM & 0.100536 (0.000080) & 0.100462 & 0.000073 & 479.57 & 0.946 (0.015) & -15.6\%\\
\hspace{1em}Tyler's M & 0.100723 (0.000079) & 0.100685 & 0.000038 & 11565.28 & 1.1302 (0.0095) & 1935.8\%\\
\addlinespace[0.3em]
\multicolumn{7}{l}{\textbf{$c_p=0.1, \rho=10000$}}\\
\hspace{1em}SSCM & 0.099559 (0.000084) & 0.099500 & 0.000058 & 573.55 & - & -\\
\hspace{1em}Latent SCM & 0.100590 (0.000089) & 0.100500 & 0.000091 & 484.30 & 1.026 (0.021) & -15.6\%\\
\hspace{1em}Tyler's M & 0.100776 (0.000085) & 0.100722 & 0.000054 & 11571.41 & 1.208 (0.010) & 1917.5\%\\
\addlinespace[0.3em]
\multicolumn{7}{l}{\textbf{$c_p=0.5, \rho=8$}}\\
\hspace{1em}SSCM & 0.49881 (0.00018) & 0.49882 & -0.00001 & 110.14 & - & -\\
\hspace{1em}Latent SCM & 0.50147 (0.00022) & 0.50118 & 0.00029 & 100.14 & 0.529 (0.026) & -9.1\%\\
\hspace{1em}Tyler's M & 0.50315 (0.00018) & 0.50318 & -0.00003 & 6576.61 & 0.8614 (0.0041) & 5871.2\%\\
\addlinespace[0.3em]
\multicolumn{7}{l}{\textbf{$c_p=0.5, \rho=100$}}\\
\hspace{1em}SSCM & 0.49740 (0.00038) & 0.49769 & -0.00029 & 110.04 & - & -\\
\hspace{1em}Latent SCM & 0.50243 (0.00042) & 0.50231 & 0.00012 & 99.98 & 1.002 (0.038) & -9.1\%\\
\hspace{1em}Tyler's M & 0.50401 (0.00038) & 0.50431 & -0.00030 & 6532.00 & 1.312 (0.011) & 5836.0\%\\
\addlinespace[0.3em]
\multicolumn{7}{l}{\textbf{$c_p=0.5, \rho=10000$}}\\
\hspace{1em}SSCM & 0.49719 (0.00040) & 0.49750 & -0.00031 & 112.45 & - & -\\
\hspace{1em}Latent SCM & 0.50247 (0.00046) & 0.50250 & -0.00002 & 101.83 & 1.051 (0.036) & -9.4\%\\
\hspace{1em}Tyler's M & 0.50410 (0.00040) & 0.50450 & -0.00040 & 6740.00 & 1.370 (0.012) & 5893.6\%\\
\addlinespace[0.3em]
\multicolumn{7}{l}{\textbf{$c_p=0.9, \rho=8$}}\\
\hspace{1em}SSCM & 0.89801 (0.00033) & 0.89796 & 0.00005 & 63.02 & - & -\\
\hspace{1em}Latent SCM & 0.90253 (0.00045) & 0.90222 & 0.00031 & 56.83 & 0.500 (0.032) & -9.8\%\\
\hspace{1em}Tyler's M & 0.92039 (0.00035) & 0.92023 & 0.00015 & 22801.89 & 2.431 (0.012) & 36082.4\%\\
\addlinespace[0.3em]
\multicolumn{7}{l}{\textbf{$c_p=0.9, \rho=100$}}\\
\hspace{1em}SSCM & 0.89604 (0.00073) & 0.89593 & 0.00011 & 64.32 & - & -\\
\hspace{1em}Latent SCM & 0.90372 (0.00085) & 0.90425 & -0.00054 & 57.88 & 0.850 (0.050) & -10.0\%\\
\hspace{1em}Tyler's M & 0.92271 (0.00074) & 0.92227 & 0.00044 & 22955.86 & 2.891 (0.019) & 35590.4\%\\
\addlinespace[0.3em]
\multicolumn{7}{l}{\textbf{$c_p=0.9, \rho=10000$}}\\
\hspace{1em}SSCM & 0.89618 (0.00066) & 0.89559 & 0.00059 & 62.94 & - & -\\
\hspace{1em}Latent SCM & 0.90552 (0.00082) & 0.90459 & 0.00094 & 56.87 & 1.031 (0.046) & -9.6\%\\
\hspace{1em}Tyler's M & 0.92338 (0.00067) & 0.92260 & 0.00078 & 22019.33 & 2.945 (0.023) & 34885.3\%\\*
\end{longtable}
\endgroup{}

\end{document}